\documentclass[12pt,a4paper]{amsart}
\usepackage{amssymb,amsthm}
\usepackage{amscd}
\usepackage{mathrsfs}
\usepackage{amsfonts}
\usepackage{graphicx}
\usepackage{amsmath}
\usepackage{amssymb}
\usepackage{amsfonts}
\usepackage{verbatim}
\usepackage{amsmath,amscd,stmaryrd,amssymb,array, amsfonts,amsmath,amssymb}
\usepackage{enumitem}
\usepackage{amsthm}
\usepackage{bm}
\usepackage{setspace}
\usepackage{nameref}
\usepackage{varioref}
\usepackage[colorlinks=true,urlcolor=blue,
citecolor=red,linkcolor=blue,linktocpage,pdfpagelabels,
bookmarksnumbered,bookmarksopen,pagebackref]{hyperref}
\usepackage{cleveref}
\usepackage{array}
\usepackage{xcolor}
\usepackage{babel}

\usepackage{latexsym}
\usepackage{subfigure}
\usepackage{tikz}
\usepackage{float}
\usepackage{graphicx,color}
\usepackage{fullpage}
\usepackage{epsfig,graphics,color}
\usepackage{graphicx}%
\newtheorem{theorem}{Theorem}[section]
\theoremstyle{plain}

\newtheorem{lemma}[theorem]{Lemma}
\newtheorem{proposition}[theorem]{Proposition}
\newtheorem{remark}{Remark}
\numberwithin{equation}{section}

\def\r{\mathbb{R}}

\def\N{\mathbb{N}}

\def\cV{\mathbb{V}}

\def\eps{\varepsilon}

\def\tilde{\widetilde}

\def\cH{\mathcal{H}}

\def\cM{\mathcal{M}}
\def\cN{\mathcal{N}}

\def\cL{\mbox{\scalebox{1.1}{$\mathcal{L}$}}}

\def\bar{\overline}

\def\setminus{\smallsetminus}
\def\implies{\Rightarrow}

\newcommand{\dint}{\displaystyle\int}

\newcommand{\what}{\widehat}

\def\eps{\varepsilon}

\def\tilde{\widetilde}

\def\cH{\mathcal{H}}

\def\cM{\mathcal{M}}
\def\N{\mathbb{N}}

\def\cV{\mathcal{V}}

\def\bar{\overline}

\def\implies{\Rightarrow}

\newcommand{\R}{\mathbb{R}}

\def\be{\begin{equation}}
	\def\ee{\end{equation}}
\def\bes{\begin{equation*}}
	\def\ees{\end{equation*}}
\def\bsp{\begin{split}}
	\def\esp{\end{split}}

\def\ba{\begin{array}}
	\def\ea{\end{array}}
\def\benu{\begin{enumerate}}
	\def\eenu{\end{enumerate}}
\def\bt{\begin{theorem}}
	\def\et{\end{theorem}}
\def\bp{\begin{proposition}}
	\def\ep{\end{proposition}}
\def\bl{\begin{lemma}}
	\def\el{\end{lemma}}
\def\br{\begin{remark}}
	\def\er{\end{remark}}
\def\bd{\begin{definition}}
	\def\ed{\end{definition}}

\def\.{\cdot}

\def\R{\mathbb{R}}

\def\~{\tilde}
\def\8{\infty}

\begin{document}

\title[A New Approach to Inspect Weakly Coupled Logistic  Systems and their Asymptotic Behavior]{A New Approach to Inspect Weakly Coupled Logistic  Systems and their Asymptotic Behavior}
\author{Haoyu Li}
\address{Haoyu Li \\ Department of Mathematics, Faculty of Science and Technology,
	University of Macau, Taipa, Macau, China.}
\email{hyli1994@hotmail.com}
\author{Liliane Maia}
\address{Liliane Maia \\ Departamento de Matem\'{a}tica, UNB, 70910-900, Brasilia, Brazil.}
\email{lilimaia@unb.br}
\author{Mayra Soares}
\address{Mayra Soares \\ Departamento de Matem\'{a}tica, UNB, 70910-900, Brasilia, Brazil.}
\email{mayra.soares@unb.br}

\rule{15.3cm}{0.001cm} 
\begin{changemargin}{-1.1cm}{-1.1cm}  
\begin{abstract}
{\small 	We consider the weakly coupled elliptic system of logistic type,
\begin{equation}
	\begin{cases}
		-\Delta u &=\mu_1u- |u|^{p-2}u+ \beta |u|^{\frac{p}{2}-2}u |v|{^{\frac{p}{2}}}\mbox{ in }\Omega,\\
		-\Delta v & =\mu_2 v- |v|^{p-2}v+\beta |u|^{\frac{p}{2}}|v|^{\frac{p}{2}-2}v \mbox{ in }\Omega,\\
	\ \	u,v &\in H_0^1(\Omega),
	\end{cases} 
\end{equation}
	where $\Omega\subset\mathbb{R}^N$ is a bounded connected domain with $N\geq 2$, $2< p < 2^*$,   $\lambda_1(\Omega)< \mu_1\leq\mu_2$, and $\lambda_1(\Omega)$ denotes the first eigenvalue of the Laplacian operator in $\Omega$ with Dirichlet boundary conditions. We say the system is competitive if $\beta<0$ and cooperative if $\beta>0$, for $\beta \in \mathbb{R}$.
	
We prove the existence and multiplicity of solutions to the problem \eqref{LS} in alternative variational frameworks, depending on the range of the parameter $\beta.$  We do not rely on bifurcation or degree theory, which have been used in the literature for logistic-type problems. Instead, the novelty is to obtain min-max type solutions by exploiting the different geometry of the functional associated with the logistic problem. In case $N\geq 2$,  we establish existence and multiplicity results for different ranges of the coupling parameter, considering both the case \(2<p<4\) and the classical case \(p=4\). We also investigate the asymptotic behavior of selected solutions as \(\beta\to1^-\) and \(\beta\to-\infty\).}
\newline	
\rule{15.3cm}{0.001cm}
\end{abstract}
\end{changemargin}
\maketitle 
\vspace{0.5cm}
\section{Introduction}
\label{sec:introduction}
We study the weakly coupled elliptic system of the logistic type 
\begin{equation}\label{LS}
	\begin{cases}
	-\Delta u &=\mu_1u- |u|^{p-2}u+ \beta |u|^{\frac{p}{2}-2}u |v|{^{\frac{p}{2}}}\mbox{ in }\Omega,\\
	-\Delta v & =\mu_2 v- |v|^{p-2}v+\beta |u|^{\frac{p}{2}}|v|^{\frac{p}{2}-2}v \mbox{ in }\Omega,\\
	\ \	u,v &\in H_0^1(\Omega),
\end{cases}
\end{equation}
	where $\Omega\subset\mathbb{R}^N$ is a bounded  connected domain with $\partial \Omega \in \mathcal{C}^2$, $N\geq 2$, $2< p < 2^*$,   $\lambda_1(\Omega)< \mu_1\leq\mu_2$, and $\lambda_1(\Omega)$ denotes the first eigenvalue of the Laplacian operator in $\Omega$ with Dirichlet boundary conditions. Our motivation comes from the particular model
\begin{equation}\label{PartSystem}
	\left\{
	\begin{array}{lr}
		-\Delta u =\mu_1u- u^{3}+ \beta u v^{2}\mbox{ in }\Omega,\\
		-\Delta v =\mu_2 v- v^{3}+\beta u^{2}v \mbox{ in }\Omega,\\
		u,v\in H_0^1(\Omega),
	\end{array}
	\right.
\end{equation}
where $\Omega\subset\mathbb{R}^3$ is a bounded domain. Such a system appears in a vast context of population dynamics and has been extensively studied in the past decades, as in \cite{CC,DD,LG,TZQW} and the references therein. System \eqref{LS} characterizes the equilibrium state of two interacting species within the environment $\Omega.$ The interaction may be either competitive if $\beta<0$ or cooperative if $\beta>0$, for $\beta \in \mathbb{R}$.

Our goal is to prove the existence of solutions to the problem \eqref{LS} in different variational frameworks depending on the information we have about the parameter $\beta.$ More specifically, we obtain at least one solution for $\beta$ in the whole line and for suitable values of $p$, as $N\geq 2$. Moreover, in the classical case $N=3$ and $p=4$, we construct a solution for all $\beta > 0$, for small $|\beta|$, and, beyond a certain threshold, for $\beta < 0$, possibly covering the entire real line. In all cases, we do not rely on non-variational methods such as sub-supersolution techniques, bifurcation theory, or degree theory, which are commonly used in the literature for logistic-type problems. Instead, our arguments are innovative and unconventional, as the geometry of the functional associated with the logistic problem prevents the direct application of standard variational methods to obtain min-max type solutions. Furthermore, we analyze the asymptotic behavior of such solutions as $\beta \to 1^-$ or $\beta \to - \infty.$

Many authors have been working on logistic problems in the last years.
 They have applied a variety of techniques involving topological methods to obtain positive solutions.
 	 Only recently the variational approach has begun to be applied. This has allowed researchers to explore more information about the existence of solutions for the ranges of coupling parameters for which sign-changing solutions are detected and positive solutions do not exist.

Some classical and pioneering works found in the literature have motivated our study. In \cite{CC}, by the Crandall-Rabinowitz bifurcation, Cantrell and Cosner established conditions on the coefficients to obtain positive solutions to the non-variational elliptic system \begin{equation}\begin{cases}-\Delta u=au-u^2-cuv,\\ -\Delta v=dv-v^2-euv,\end{cases} \end{equation}
		in a bounded domain with the Dirichlet boundary condition.
Likewise, in \cite{DD}, Dancer and Du 
	 obtained positive solutions to the same problem when the interaction parameters $c$ and $e$ are large, by applying a degree theory argument.  We also cite \cite[Chapter 10]{LG}, by López-Gómez, for an extensive survey on general non-variational logistic systems. 
	
More recently, relying on local and global bifurcations in \cite{TZQW}, Tian and Wang identified some intervals of $\beta$ in which there exist finitely many bifurcation points  with respect to a trivial solution branch to the system
\begin{equation}\label{VS}
	\begin{cases}-\Delta u=au-\mu_1u^3+\beta cuv^2, \ u>0\\ -\Delta v=av-\mu_2v^3+\beta u^2v, \ v>0 \end{cases}\end{equation} 
		in a bounded domain, with Dirichlet conditions, where $a > \lambda_1(\Omega)$. Working with a similar system, we extend and complement their results, covering a larger range on the parameter $\beta$ for the logistic case. Due to our approach, in addition to positive solutions, we are also able to find seminodal solutions.
	
 Synchronized solutions for the system \eqref{VS} were constructed by Cheng and Zhang in	\cite{CZ2021}. Considering different parameters $\beta_1$ and $\beta_2$, instead of $\beta$, the authors presented a detailed study of existence and uniqueness, using comparison arguments. In the cases where our problem coincides with theirs, our results establish the multiplicity of solutions.  Furthermore,  for certain ranges of $\beta$, we obtain additional solutions that were not previously identified.

We consider $\cH := H_0^1(\Omega)\times H_0^1(\Omega)$, endowed with the norm given by \[\|(u,v)\|^2 := \|\nabla u\|_2^2 + \|\nabla v\|_2^2, \mbox{ for every } (u,v)\in \cH.\] Then, the $\mathcal{C}^1$ functional $J_\beta: \cH \to \r$, associated with problem \eqref{LS}, is given by 
\begin{equation}\label{functionalbeta}
	J_\beta(u,v)=\frac{1}{2}\dint_{\Omega}|\nabla u|^2-\mu_1u^2+|\nabla v|^2-\mu_2 v^2+\frac{1}{p}\dint_{\Omega} |u|^p+|v|^p-\frac{2\beta}{p}\dint_{\Omega} |u|^{\frac{p}{2}}|v|^{\frac{p}{2}}.
\end{equation}
Especially, if $\beta = 0$,
\begin{equation}\label{functionalzero}
	J_0(u,v)=\frac{1}{2}\dint_{\Omega}|\nabla u|^2-\mu_1u^2+|\nabla v|^2-\mu_2 v^2+\frac{1}{p}\dint_{\Omega} |u|^p+|v|^p.
\end{equation}
Its derivative is given by 
\begin{align}\label{J'}
	J'_\beta(u,v)(\varphi,\psi)=&\dint_{\Omega}\nabla u\cdot \nabla \varphi + \nabla v\cdot \nabla \psi -\mu_1u\varphi -\mu_2 v\psi\nonumber\\
	&+\dint_{\Omega} |u|^{p-2}u\varphi+|v|^{p-2}v\psi- \beta( |u|^{\frac{p}{2}-2}|v|^{\frac{p}{2}}u\varphi + |u|^{\frac{p}{2}}|v|^{\frac{p}{2}-2}v\psi ),
\end{align}
where $(\varphi, \psi) \in \cH.$ It is known that the critical points of $J_\beta$ are the weak solutions to problem \eqref{LS}. A critical point $(u,v)$ is a vectorial solution if $ u \neq 0$ and $v \neq 0$, and it is called a semi-trivial solution if either $u = 0$ and $v\neq 0$, or $u\neq 0$ and $v = 0$. Moreover, we say a vectorial solution is semi-nodal if one of the components changes sign. From a biological perspective, the extinction of a species is represented by the vanishing of the corresponding component. Thus, semi-trivial solutions can be interpreted as coexistence states in which one species persists while the other becomes extinct, indicating that one species dominates the other.

We say that a critical point $(u,v)$ is of positive energy when $J_\beta(u,v)>0$. Similarly, for negative energy.  
We introduce the Nehari set
\begin{align}\label{nehari}
	\mathcal{N}_\beta=\left\{(u,v)\in \cH \setminus \{(0,0)\} :\begin{array}{ll}& \dint_{\Omega}|\nabla u|^2-\mu_1u^2+|\nabla v|^2-\mu_2 v^2 \\&+\dint_{\Omega} |u|^p+|v|^p-2\beta\dint_{\Omega} |u|^{\frac{p}{2}}|v|^{\frac{p}{2}}=0\end{array}\right\}.
\end{align}
Moreover, we define the minimum energy level on the Nehari manifold by,
\begin{equation}\label{infimum}
m_\beta:=\inf_{(u,v)\in\mathcal{N}_\beta}J_\beta(u,v).
\end{equation}

Our main results are stated as follows. 

\begin{theorem}\label{t:beta(0,1)}
	Assume  $\mu_2\geq\mu_1>\lambda_1(\Omega)$. If $N\geq2$ and $ 2< p < 2^*$, then there exists a constant $\delta_0$ such that for any $\beta\in(-\delta_0,1)$, the ground state $m_\beta$ is attained by a non-negative vectorial solution $(u_\beta,v_\beta)$ to problem \eqref{LS}. Moreover, $m_\beta\to-\infty$ as $\beta \to 1^-$. 
\end{theorem}
For $\beta $ negative, and away from zero, we prove the existence of a vectorial solution as a mountain-pass solution. Moreover, we study the asymptotic behavior of these solutions as $\beta \to -\infty$, obtaining a segregation phenomenon. 
In this scenario, we have been inspired by the remarkable work \cite{ContiTerraciniVerzini2002}, where  Conti, Terracini, Verzini proved the existence of ground state solutions and segregation argument for competing species. Although, they did not tackle the logistic framework, we are able to extend their type of results to our setting. Regarding the segregation effect, we take advantage of the results developed in \cite{STTZ, NorisTavarseTerraciniVerzini2010} to construct a sign-changing solution of the associated scalar problem below. We use the notation $u^+=\max \{u,0\}$ and $u^-=\max\{-u,0\}$.

\begin{theorem}\label{t:mountainpass}
	Assume  $\mu_2\geq\mu_1>\lambda_1(\Omega)$. If $N\geq2$ and $ 2< p < \min\{4,2^*\}$, then there exists a constant $\delta_\star(\beta)>0$ and a non-negative vectorial solution $(u_\beta,v_\beta)$ to problem \eqref{LS}  such that 
	\begin{equation}\label{MPenergy}
m_\beta \leq	c_1+\delta_\star<J_\beta(u_\beta,v_\beta) = c_\beta <-\delta_\star,\end{equation}
 for all $\beta<0$.

{
	In addition, in the classical case  $N=3$ and $p=4$, there exist constants ${\beta_\star},\delta_\star(\beta)>0$  and  a non-negative vectorial solution $(u_\beta,v_\beta)$ to problem \eqref{LS}  such that,  for $\beta<- {\beta_\star}$, it holds \eqref{MPenergy}. Furthermore, if $u_\infty$ and $v_\infty$ are both nontrivial, then $$\mu_1>  \lambda_1(\Omega_u), \mbox{ and }  \mu_2 > \max\{ \lambda_1(\Omega_u), \lambda_1(\Omega_v)\} \geq \lambda_2(\Omega).$$ In addition, the function $w_\infty = w_\infty^+ - w_\infty^-$ is a sign-changing solution  to
	\begin{equation}\label{e:SignChanging}
		\left\{
		\begin{array}{ll}
			-\Delta u &=\mu_1u^+ -\mu_2 u^- -u^{3}\mbox{ in }\Omega,\\
			\quad u &= u^{+}-u^-\in H_0^1(\Omega).
		\end{array}
		\right.
	\end{equation}}
\end{theorem}
\begin{remark} (i) For $2 < p <\min\{4,2^*\}$ and  $|\beta|$ sufficiently small, Theorems \ref{t:beta(0,1)} and \ref{t:mountainpass} guarantee the existence of two distinct non-negative vectorial solutions to problem \eqref{LS}. More precisely, Theorem \ref{t:beta(0,1)} yields a ground state solution with energy $m_\beta$ strictly below the minimum energy among all positive semi-trivial solutions. On the other hand, Theorem \ref{t:mountainpass} provides a second solution whose energy lies strictly above the maximum energy of the positive semi-trivial solutions.

\noindent(ii) In the special case $\mu_1=\mu_2=:\mu>\lambda_1(\Omega)$, we have the following synchronized solution
	\begin{align}
		(w_\beta,w_\beta)=\left(\frac{w}{{|\beta-1|^{\frac{1}{p-2}}}},\frac{w}{|\beta-1|^{\frac{1}{p-2}}}\right), \quad \mbox{for every }  \ \beta<1, \nonumber
	\end{align}
	where $w$ is the unique positive solution of
	\begin{equation}\label{SP}
		\left\{
		\begin{array}{lr}
			-\Delta w =\mu w -w^{p-1}\mbox{ in }\Omega,\\
			0<w\in H_0^1(\Omega).
		\end{array}
		\right.
	\end{equation}
	This result confirms and describes the solution found in \cite{CZ2021}. Notice that 
	$
		(w_\beta,w_\beta)\to(0,0)\mbox{ in }\cH\nonumber
	$
	as $\beta \to-\infty$.
	This solution is different from the solution $(u_\beta,v_\beta)$ we obtained in Theorem \ref{t:mountainpass}, since $(u_\beta,v_\beta)$ segregates as $\beta \to-\infty$. Therefore, there exist at least two solutions in this case.

\noindent(iii) According to \cite[Theorem C]{STTZ}, as $\beta \to -\infty$, the sign-changing solution $w_\infty$ to problem \eqref{e:SignChanging}, obtained in Theorem \ref{t:mountainpass}, is Hölder continuous. Moreover, the free boundary, given by the nodal set ${w_\infty=0}$, is characterized by the optimal partition associated with the segregation limit. This nodal set consists of a collection of hypersurfaces of class $\mathcal{C}^{1,\alpha}$, together with an irregular subset $A$ of Hausdorff dimension at most $N-2$.
\end{remark}

In what follows, we present a non-existence result of positive solutions to problem \eqref{LS} with $\beta>1$. 
\begin{theorem}\label{t:nonexistencebeta>0}
	Assume that $\beta>1$ and $\lambda_1(\Omega)< \mu_1\leq\mu_2$. If $N\geq2$ and $ 2< p < 2^*$,
    then Problem \eqref{LS} does not admit a positive vectorial solution.
\end{theorem}
Moreover, we obtain a vectorial solution, which means that at least one of its component functions must change sign.

\begin{theorem}\label{t:existencebeta>1}
	Assume that $\beta>1$ and $\lambda_1(\Omega)< \mu_1\leq\mu_2$. If $N\geq2$ and $ 2< p < 2^*$, then Problem \eqref{LS} admits a vectorial solution $(u, v)$ with positive energy, and at least one of the components $u$ or $v$ changes sign.
\end{theorem}

Throughout this manuscript, $\lambda_1(\Omega)$  denotes the first eigenvalue of the Laplacian operator on $\Omega$  subject to homogeneous Dirichlet boundary conditions, while $\lambda_i(\Omega)$  denotes the $i-$th eigenvalue under the same boundary conditions.

\section{The ground state solution for \texorpdfstring{$\beta\in[-\delta_0,1)$}{beta in [-delta0,1)}}
We prove Theorem \ref{t:beta(0,1)} in some steps.
First, when $|\beta|$ is small, we prove the following proposition.
\begin{proposition}\label{prop:SmallBeta}
	There exists a constant  $\delta_0=\delta_0(\lambda_1,\mu_2,\Omega)>0$ such that if $|\beta|\leq\delta_0$, problem \eqref{LS} admits a non-negative  minimizer $(u_\beta,v_\beta)$, which is a vectorial solution.
\end{proposition}
To establish this result, we shall make use of the following lemmas.
\begin{lemma}\label{l:NehairBoundedness}
	If $\beta<1$, there exists a constant $C_\beta>0$ such that for any ${(u,v)\in\mathcal{N}_\beta}$, $\|(u,v)\|\leq C_\beta$. Moreover, if $\beta \leq 0$, then $C_\beta \leq C_0$, which is a uniform bound independent of $\beta$.
\end{lemma}
\begin{proof}
	For any $(u,v)\in\mathcal{N}_\beta$, by definition,
	\begin{align}\label{Equa:Nehari}
		\dint_{\Omega}|\nabla u|^2-\mu_1u^2+|\nabla v|^2-\mu_2 v^2 +\dint_{\Omega} |u|^p+|v|^p-2\beta\dint_{\Omega} |u|^{\frac{p}{2}}|v|^{\frac{p}{2}}=0.
	\end{align}
Using  H\"older's and Young inequalities, in case $0< \beta <1$, we deduce
	\begin{align} \label{Lpestimates}
		\dint_{\Omega} |u|^p + |v|^p\leq \dint_{\Omega}|\nabla u|^2 +|\nabla v|^2 +\dint_{\Omega} |u|^p + |v|^p\leq C(|u|_p^2+|v|_p^2)+\beta\dint_{\Omega} |u|^p + |v|^p,
	\end{align}
	for some constant $C>0$ depending on $\lambda_1,\mu_2$ and $\Omega$. 	Since $\beta<1$,
	\begin{align}
		\dfrac{1}{2^{\frac{p}{2}-1}}(|u|_p^2+|v|_p^2)^\frac{p}{2} \leq |u|_p^p+|v|_p^p\leq\frac{C}{1-\beta}(|u|_p^2+|v|_p^2) \implies |u|_p^2+|v|_p^2\leq2\left(\frac{C}{1-\beta}\right)^{\frac{1}{\frac{p}{2}-1}}.\nonumber
	\end{align}
	Using (\ref{Equa:Nehari}) again,
	\begin{align*}
		\dint_{\Omega}|\nabla u|^2+\dint_{\Omega}|\nabla v|^2\leq C(|u|_p^2+|v|_p^2)+(\beta-1)\dint_{\Omega} |u|^p + |v|^p\leq 2C\left(\frac{C}{1-\beta}\right)^{\frac{1}{\frac{p}{2}-1}}=:\mathcal{C}^2_\beta.
	\end{align*}
	In case $\beta \leq 0$, we repeat the previous estimates from \eqref{Lpestimates} with $\beta =0$, we observe that $ \mathcal{C}^2_\beta \leq \mathcal{C}^2_0$, which is independent of $\beta.$
\end{proof}

Now we analyze the set $\mathcal{N}_\beta$. 
Setting
$
G_\beta(u,v):=J_\beta'(u,v)(u,v),
$
we have that
\[
\mathcal N_\beta
=
\{(u,v)\in\mathcal H\setminus\{(0,0)\}:G_\beta(u,v)=0\}.
\]
For $\beta<1$, we notice that, in view of the Cauchy-Schwarz inequality,
\[
\begin{aligned}
B_\beta(u,v)
&:=\int_\Omega
\left(
|u|^p+|v|^p
-2\beta |u|^{p/2}|v|^{p/2}
\right)
\geq \min\{1,1-\beta\}
\int_\Omega(|u|^p+|v|^p)>0,
\end{aligned}
\]
for every $(u,v)\neq(0,0)$. If $(u,v)\in\mathcal N_\beta$, then
$
Q(u,v)+B_\beta(u,v)=0,
$
where
\[
Q(u,v):=\int_\Omega
\left(
|\nabla u|^2-\mu_1u^2+
|\nabla v|^2-\mu_2v^2
\right).
\]
Using the homogeneity of $Q$ and $B_\beta$, we obtain
\[
G_\beta'(u,v)(u,v)
=2Q(u,v)+pB_\beta(u,v)
=(p-2)B_\beta(u,v)>0.
\]
Hence $G_\beta'(u,v)\neq0$ on $\mathcal N_\beta$, and the implicit
function theorem implies that $\mathcal N_\beta$ is a $\mathcal{C}^1$ manifold
of co-dimension one.
Moreover, if $(u,v)\in\mathcal N_\beta$ is a critical point of
$J_\beta|_{\mathcal N_\beta}$, then there exists $\lambda\in\mathbb R$
such that
\[
J_\beta'(u,v)=\lambda G_\beta'(u,v).
\]
Testing this identity against $(u,v)$, it gives
$
0=\lambda G_\beta'(u,v)(u,v).
$
Since $G_\beta'(u,v)(u,v)>0$, we conclude that $\lambda=0$ and hence
$
J_\beta'(u,v)=0.
$
Therefore, $\mathcal N_\beta$ is a natural constraint for
$J_\beta$.

Now, let us define the set of projectable functions on the Nehari manifold $\cN_\beta.$
\begin{align}\label{L_}
	\cL^-:=\left\{(u,v)\in \cH:Q(u,v)<0\right\}.\nonumber
\end{align}

\begin{lemma}\label{projection}
	If $\beta<1$, for any $(u,v)\in \cL^-$ there exists $\what t\in \R^+$ given by 
	\begin{equation}\label{tuv}
		\what t :=  t_{u,v,\beta}=\left(\frac{-\dint_{\Omega}|\nabla u|^2-\mu_1u^2+|\nabla v|^2-\mu_2 v^2}{\dint_{\Omega} |u|^p + |v|^p-2\beta |u|^{\frac{p}{2}}|v|^{\frac{p}{2}}}\right)^\frac{1}{p-2}
	\end{equation}
	such that $\what t (u, v)\in\mathcal{N}_\beta$. Furthermore, $m_\beta$ can be characterized by
	\begin{align}\label{m_beta}
		m_\beta=\inf_{(u,v)\in \mathcal{L}^-}J_\beta(\what t(u,v))=\inf_{(u,v)\in \mathcal{L}^-}-\dfrac{p-2}{2p}\dfrac{\left(-\dint_{\Omega}|\nabla u|^2-\mu_1u^2+|\nabla v|^2-\mu_2 v^2\right)^{\frac{p}{p-2}}}{\left(\dint_{\Omega} |u|^p + |v|^p-2\beta |u|^{\frac{p}{2}}|v|^{\frac{p}{2}}\right)^\frac{2}{p-2}}.
	\end{align}
\end{lemma}
\begin{proof}
	Standard calculations using \eqref{nehari} yield the existence of $\what t\in \R^+$ given by \eqref{tuv}. Furthermore,
	\begin{align}\label{eq:Jbetanehari}
		J_\beta(\what t(u,v))&=\frac{\what t ^2}{2}\dint_{\Omega}|\nabla u|^2-\mu_1u^2+|\nabla v|^2-\mu_2 v^2+\frac{\what t^p}{p}\dint_{\Omega} (|u|^p + |v|^p-2\beta |u|^{\frac{p}{2}}|v|^{\frac{p}{2}})\nonumber\\
		&= \dfrac{\what t^2}{2}  \left( \dint_{\Omega}|\nabla u|^2-\mu_1u^2+|\nabla v|^2-\mu_2 v^2 +\frac{2\what t^{p-2}}{p}\dint_{\Omega} (|u|^p + |v|^p-2\beta |u|^{\frac{p}{2}}|v|^{\frac{p}{2}})  \right)\nonumber\\
		&	=-\dfrac{p-2}{2p}\dfrac{\left(-\dint_{\Omega}|\nabla u|^2-\mu_1u^2+|\nabla v|^2-\mu_2 v^2\right)^{\frac{p}{p-2}}}{\left(\dint_{\Omega} |u|^p + |v|^p-2\beta |u|^{\frac{p}{2}}|v|^{\frac{p}{2}}\right)^\frac{2}{p-2}}.
	\end{align}
\end{proof}

We denote by $w_i$ the unique positive solution of
\begin{equation}\label{P_scalar}
	\left\{
	\begin{array}{lr}
		-\Delta u =\mu_i u -u^{p-1}\mbox{ in }\Omega,\\
		0<u\in H_0^1(\Omega)
	\end{array}
	\right. \mbox{ \ for \ } i=1,2.
\end{equation}
for $i=1,2$, respectively. Defining 
\begin{align}
J_i(u)=\frac{1}{2}\dint_{\Omega} |\nabla u|^2-\mu_i u^2+\frac{1}{p}\dint_{\Omega} |u|^p,\nonumber
\end{align}
associated to the scalar problem \eqref{P_scalar}, we set $c_i:=J_i(w_i)$. 
We refer to \cite{CardosoFurtadoMaia2024} for further information.
We observe that for $\beta =0$, it holds that $J_0(u,v) = J_1(u)+ J_2(v)$ and so, 
$t_{w_1,w_2,0}=1$. 

We recall that by the well known results in the literature, \cite{AT, BZ, LG, Ouy} we have the following result.

\begin{lemma}\label{eq:Pi}
	Assume that  $\lambda_1(\Omega)< \mu_1\leq\mu_2$. It holds that
	\begin{itemize}
		\item [$(i)$] $(u_1,0)$ and $(0,u_2)$ are semi-trivial solutions to problem \eqref{LS}, where $u_i$ solves
		\eqref{P_scalar};		
		
		\item [$(ii)$] $J_\beta(w_1,0)=J_1(w_1)=:c_1<0$ and $J_\beta(0,w_2)=J_2(w_2)=:c_2\leq c_1<0$;
		\item [$(iii)$] $(w_1,0)$ and $(0,w_2)$ are the only semi-trivial non-negative solutions to problem \eqref{LS};
		\item[(iv)] any semi-trivial critical point of $J_\beta$ has negative energy.
	\end{itemize}
\end{lemma}

\begin{lemma}\label{l:tcontinuous}
	The continuity of $t_{u,v,\beta}$ in $\beta$ yields   $|t_{w_1,w_2,\beta}-1|=o_\beta(1)$ as $\beta\to0$. Here, $o_\beta(1)$ is the infinitesimal as $\beta\to0$.
\end{lemma}

As a consequence, we obtain the next lemma.
\begin{lemma}\label{l:nontrivialenergy}
	There exists $\delta_0>0$ such that, if $|\beta|\leq\delta_0$, then \[m_\beta <\min\{c_1,c_2\}.\]
\end{lemma}
\begin{proof}
	In fact, by continuity of $ \beta \mapsto J_\beta$, Lemma \ref{l:tcontinuous}, and the fact that $c_i<0$ for $i=1,2$, we obtain
	\begin{align}
		J_\beta(t_{w_1,w_2,\beta}w_1,t_{w_1,w_2,\beta}w_2)\nonumber&=J_0(t_{w_1,w_2,\beta}w_1,t_{w_1,w_2,\beta}w_2)+o_\beta(1)\nonumber\\
		&=J_0(w_1,w_2)+o_\beta(1)=J_1(w_1)+J_2(w_2)+o_\beta(1)\nonumber\\
		&=c_1+c_2+o_\beta(1)<\min\{c_1,c_2\}.\nonumber
	\end{align}
\end{proof}

\begin{proof}[\bf Proof of Proposition \ref{prop:SmallBeta}]
	Since $m_\beta$ is finite, there exists a minimizing sequence $(u_n,v_n)$ such that $J_\beta(u_n,v_n) \to m_\beta$ and by standard arguments we obtain a minimizer $(u_\beta,v_\beta)$, which is a vectorial solution, by Lemma \ref{l:nontrivialenergy}. Moreover, since $(|u_\beta|,|v_\beta|)$ is also a minimizer, without loss of generality we may assume $u_\beta \gvertneqq 0$ and $v_\beta \gvertneqq 0$.
\end{proof}

Now, when $\beta\in(0,1)$, we prove the following monotonicity result. 

\begin{lemma}\label{l:monotone}
	Assume $0<\beta_1,\beta_2<1$, then it holds that $m_{\beta_1},m_{\beta_2}<\infty$. If $\beta_1\leq\beta_2$, we get $m_{\beta_1}\geq m_{\beta_2}$. More precisely, the map $\beta \mapsto m_\beta$ is non-increasing. 
\end{lemma}
\begin{proof}
	By Lemma \ref{projection}, the number $m_\beta$ can be rewritten as
	\begin{align}
		m_\beta=\inf_{(u,v)\in \mathcal{L}^-}-\dfrac{p-2}{2p}\dfrac{\left(-\dint_{\Omega}|\nabla u|^2-\mu_1u^2+|\nabla v|^2-\mu_2 v^2\right)^{\frac{p}{p-2}}}{\left(\dint_{\Omega} |u|^p + |v|^p-2\beta |u|^{\frac{p}{2}}|v|^{\frac{p}{2}}\right)^\frac{2}{p-2}}.\nonumber
	\end{align}
	Let $\beta_1,\beta_2\in(0,1)$ with $\beta_1\leq\beta_2$.
	For any fixed $(u,v)\in \cL^-$, it is easy to verify that
	\begin{align}
	-\dfrac{\left(-\dint_{\Omega}|\nabla u|^2-\mu_1u^2+|\nabla v|^2-\mu_2 v^2\right)^{\frac{p}{p-2}}}{\left(\dint_{\Omega} |u|^p + |v|^p-2\beta_1  |u|^{\frac{p}{2}}|v|^{\frac{p}{2}}\right)^\frac{2}{p-2}}\geq -\dfrac{\left(-\dint_{\Omega}|\nabla u|^2-\mu_1u^2+|\nabla v|^2-\mu_2 v^2\right)^{\frac{p}{p-2}}}{\left(\dint_{\Omega} |u|^p + |v|^p-2\beta_2  |u|^{\frac{p}{2}}|v|^{\frac{p}{2}}\right)^\frac{2}{p-2}}.\nonumber
	\end{align}
	The result follows immediately.
\end{proof}

\begin{lemma}\label{positivesol}
	For any $\beta\in(0,1)$, $m_\beta<\min\{c_1,c_2\}$. In particular,   $m_\beta$ is attained by a positive vectorial solution $(u_\beta,v_\beta)$.
\end{lemma}
\begin{proof}
	By Lemma \ref{l:nontrivialenergy}, there exists a constant $\delta_0$ such that if $|\beta|\leq\delta_0$, the result holds. Otherwise, if $\beta\in(\delta_0,1)$, by Lemma \ref{l:monotone}, we know that
	\[
	m_\beta\leq m_{0}<\min\{c_1,c_2\}.\nonumber
	\]
	Following the same argument in the proof of Proposition \ref{prop:SmallBeta}, we guarantee the existence of a non-negative minimizer. Since $\beta >0$ we can apply the weak maximum principle and conclude that the solution is positive. 
\end{proof}

By the previous arguments, we get the asymptotic behavior of the ground state level. 

\begin{proposition}\label{prop:betato1}
	$m_\beta\to-\infty$ as $\beta \to 1^-$.
\end{proposition}
\begin{proof}
 Let $\phi_1$ be the first eigenfunction associated to $\lambda_1(\Omega)$.	Notice that $(\phi_1,\phi_1)\in \cL^-$. Then, by \eqref{m_beta} we get
	\begin{align}
	m_\beta\leq
-\frac{p-2}{2p}
\frac{
\left[
(\mu_1+\mu_2-2\lambda_1(\Omega))
\|\phi_1\|_2^2
\right]^{\frac p{p-2}}
}{
\left[
2(1-\beta)\|\phi_1\|_p^p
\right]^{\frac2{p-2}}
} \to -\infty \quad \mbox{as} \quad \beta \to 1^-.\nonumber
	\end{align}
	The result follows immediately.
\end{proof}

\begin{proof}[Proof of Theorem \ref{t:beta(0,1)}]
Combining  Lemma \ref{positivesol} and Propositions \ref{prop:betato1} and  \ref{prop:SmallBeta}, the result follows.
\end{proof}

\section{A mountain-pass approach to the vectorial solutions to problem \texorpdfstring{\ref{LS}}{ref{LS}} with \texorpdfstring{$\beta<0$}{beta<0}}

In what follows, we study the existence of vectorial solutions in the case  of strong competition, that is,  $\beta<0$ away from zero. Inspired by the approach developed in \cite{AmbrosettiColorado2007} to define the mountain pass structure
and looking for a  non-negative vectorial solution, we define the following modified functional:
\begin{align}
	J_\beta^+(u,v)& =\frac{1}{2}\dint_{\Omega}|\nabla u|^2-\mu_1|u^+|^2+|\nabla v|^2-\mu_2|v^+|^2+\frac{1}{p}\dint_{\Omega} |u|^p + |v|^p  -\frac{2\beta}{p}\dint_{\Omega} |u|^{\frac{p}{2}}|v|^{\frac{p}{2}}.\nonumber
\end{align}
Notice that $J_\beta^+$ is a $\mathcal{C}^1$-functional associated to the system 
\begin{equation}\label{system+}
	\left\{
	\begin{array}{lr}
		-\Delta u =\mu_1u^+- |u|^{p-2}u+ \beta |u|^{\frac{p}{2}-2}u |v|{^{\frac{p}{2}}}\mbox{ in }\Omega,\\
		-\Delta v =\mu_2 v^+- |v|^{p-2}v+ \beta |u|^{\frac{p}{2}} |v|{^{\frac{p}{2}-2}}v\mbox{ in }\Omega,\\
		u,v\in H_0^1(\Omega),
	\end{array}
	\right.
\end{equation}
and also that, in view of the Maximum Principle, any vectorial critical point $(u,v)$ of $J_\beta^+$ is a non-negative vectorial solution to problem \eqref{LS}. 

\begin{remark}
	For the case $\mu_1=\mu_2>\lambda_2(\Omega)$, the existence of a sign-changing solution of problem (\ref{e:SignChanging}) was proved in \cite{CardosoFurtadoMaia2024}. From a biological point of view, the sign-changing solutions describe the segregation of two species in strong competition.
\end{remark}

The following proposition proves that $(w_1,0)$ and $(0,w_2)$ are strict local minima of the functional $J_\beta^+$, giving us a mountain pass geometry for our functional.

\begin{proposition}\label{p:mountainpasslevel}
	Assume $\mu_2\geq\mu_1>\lambda_1(\Omega)$. For $N\geq2$, $ 2< p < \min\{4,2^*\}$, and for each $\beta < 0 $, there exist constants  $\varepsilon_\star, C_\star>0$  such that,  for $\varepsilon \in (0,\varepsilon_\star]$, it follows that  $J_\beta^+(w_1+\varepsilon e_1,\varepsilon e_2)\geq c_1+ C_\star\varepsilon^2$ and $J_\beta^+(\varepsilon e_1,w_2+\varepsilon e_2)\geq c_2+C_\star\varepsilon^2$ for any $(e_1,e_2)\in \cH$ with $\|e_1\|^2+\|e_2\|^2=1$.
Moreover, if $N=3$ and $p=4$, there exist constants $\varepsilon_\star, C_\star>0$ and $\beta_\star > 0, $  such that  for $\varepsilon \in (0,\varepsilon_\star]$ and $\beta < -\beta_\star$, the same assertion holds. 
\end{proposition}

In order to prove this proposition, we are going to use the results in \cite{LiYau1983} and state the following theorem.
\begin{theorem}\label{t:LiYau}
	Let $N\geq3$, $V\in L^{\frac{N}{2}}(\Omega)$, $V_-(x)=\min \{0,V(x)\}$, and $n(0)$ denotes the number of non-positive eigenvalues of the operator $-\Delta+V(x)$. Then there exists a constant $C(N)>0$  such that
	\begin{align}\label{Ineq:LiYau}
		n(0)\leq C(N)\dint_\Omega|V_-|^\frac{N}{2}.
	\end{align}
\end{theorem}
\begin{proof}
	In \cite[Theorem 2]{LiYau1983} a lower bound is established for the $k^{th}-$eigenvalue for the auxiliary problem \begin{equation}\label{AP} \tag{$AP$} \begin{cases} -\Delta \psi (x) &= \mu q(x) \psi(x),  \  \mbox{in} \
		D\\
		\ \ \ \ \  \psi(x) &= 0, \ \ \mbox{on} \ \ \partial D.
	\end{cases}\end{equation}
	Then, in \cite[Corollary 2 (iv)-(v)]{LiYau1983} an upper bound for $n(0)$ is provided by an equivalence between the number of eigenvalues less than one to problem \eqref{AP} and  $n(0)$.
\end{proof}
Now we are ready to prove Proposition \ref{p:mountainpasslevel}.
\begin{proof}[{\bf Proof of Proposition \ref{p:mountainpasslevel}}]
{We only prove the assertion concerning
$(w_1,0)$, namely, we find $C_\star>0$ such that	$J_\beta^+(w_1+\varepsilon e_1,\varepsilon e_2)\geq c_1+ C_\star\eps^2$ for $\|e_1\|^2+\|e_2\|^2=1$, $\eps \in (0,\eps_\star)$. The proof for $(0,w_2)$
is analogous. We first consider
$2<p<4$ and an arbitrary fixed $\beta<0$.
The case $p=4$ will be treated separately,
for sufficiently negative $\beta$.	}
  Let us start observing that   $w_1>0$ in $\Omega$, then 
  {\small		\begin{align}\label{eq:J+beta}
	J^+_\beta&(w_1+\varepsilon e_1,\varepsilon e_2)\geq J_\beta(w_1+\varepsilon e_1,\varepsilon e_2)\nonumber=\frac{1}{2}\dint_{\Omega} |\nabla(w_1+\varepsilon e_1)|^2-\mu_1|w_1+\varepsilon e_1|^2
		 \nonumber\\
		&+\frac{1}{2}\dint_{\Omega}|\nabla(\varepsilon e_2)|^2-\mu_2|\varepsilon e_2|^2 +\frac{1}{p} \dint_{\Omega} |w_1+\varepsilon e_1|^p +|\varepsilon e_2|^p 	+\frac{2|\beta|}{p}\dint_{\Omega}|w_1+\varepsilon e_1|^{\frac{p}{2}} | \eps e_2|^{\frac{p}{2}}\nonumber\\
		&\geq  \ J_1(w_1+\varepsilon e_1)+\frac{\varepsilon^2}{2}\dint_{\Omega}( |\nabla e_2|^2  -\mu_2e_2^2)+\frac{2|\beta|}{p}\dint_{\Omega}\left|w_1 -  \eps |e_1|\right|^{\frac{p}{2}}  |\eps e_2|^\frac{p}{2},
	\end{align}
  Using that $p/2 > 1$, and setting $\Omega_\eps: = \{x\in \Omega : w_1(x)\geq \eps|e_1(x)|\}$ we observe that 
  \begin{align*}
 |w_1 - \eps |e_1||^{\frac{p}{2}} =  
 (w_1 - \eps |e_1|)^{\frac{p}{2}}  &\geq |w_1|^\frac{p}{2} - \frac{p}{2}|w_1|^{\frac{p}{2}-1}\eps |e_1|, \ \mbox{ in } \  \Omega_\eps,\\
 |w_1 - \eps |e_1||^{\frac{p}{2}} =  
   ( \eps |e_1| - w_1 )^{\frac{p}{2}}  &\geq |\eps e_1|^\frac{p}{2}  - \frac{p}{2}|\eps e_1|^{\frac{p}{2}-1}w_1\\ & \geq  \left(1  - \frac{p}{2}\right)|\eps e_1|^{\frac{p}{2}}, \ \mbox{ in } \ \Omega \smallsetminus \Omega_\eps.
  \end{align*}
 Then, for a general constant $C>0,$
 \[\dint_{\Omega}|w_1-\varepsilon |e_1||^{\frac{p}{2}} | \eps e_2|^{\frac{p}{2}} \geq \dint_{\Omega_\eps}|w_1|^\frac{p}{2}|\eps  e_2|^{\frac{p}{2}} - C\eps^{\frac{p}{2}+1} - C\eps^{p}\geq \dint_{\Omega_\eps}|w_1|^\frac{p}{2}| \eps e_2|^{\frac{p}{2}} -  C\eps^{\frac{p}{2}+1} , \]
which combined with \eqref{eq:J+beta} yields
	\begin{align}\label{eq:J+beta2}
	J^+_\beta(w_1+\varepsilon e_1,\varepsilon e_2)
		&\geq  \ J_1(w_1+\varepsilon e_1)+\frac{\varepsilon^2}{2}\dint_{\Omega}( |\nabla e_2|^2  -\mu_2e_2^2)\notag\\& \qquad +\frac{2|\beta|\varepsilon^\frac{p}{2}}{p}\dint_{\Omega_\eps}|w_1|^\frac{p}{2}| e_2|^{\frac{p}{2}} - |\beta|C\eps^{\frac{p}{2}+1}  \nonumber\\
		&\geq\ J_1(w_1+\varepsilon e_1)+I_{\varepsilon,\beta}(e_2) - \color{red}{|\beta|C\eps^{\frac{p}{2}+1} }
	\end{align}
where, 
	\begin{align}\label{QUAD}
	I_{\eps,\beta}(e_2)&:=\frac{\varepsilon^2}{2}\dint_{\Omega}( |\nabla e_2|^2  -\mu_2e_2^2)+\frac{2|\beta|\varepsilon^\frac{p}{2}}{p}\dint_{\Omega_\eps}w_1^\frac{p}{2} |e_2|^\frac{p}{2}.
\end{align}}

	We estimate each part separately.	Firstly, using that $w_1$ solves the problem (\ref{P_scalar}) with $i=1$ and $ D^2 J_1(w_1)$ is the second order derivative of $J_1$ at $w_1$.
	\begin{align}
		J_1(w_1+\varepsilon e_1) &=\frac{1}{2}\dint_{\Omega}|\nabla w_1|^2-\mu_1|w_1|^2+\frac{1}{p}\dint_{\Omega} |w_1|^p\nonumber+\varepsilon\dint_{\Omega}\nabla w_1\cdot\nabla e_1-\mu_1w_1 e_1\nonumber\\
		&\quad +\varepsilon\dint_{\Omega} w_1^{p-1} e_1 +\frac{\varepsilon^2}{2}\dint_{\Omega}(|\nabla e_1|^2-\mu_1e_1^2)+(p-1)\frac{\varepsilon^2}{2} \dint_{\Omega} w_1^{p-2} e_1^2+{o(\varepsilon^2)}\|e_1\|^2\nonumber\\
		&=J_1(w_1)+\frac{\varepsilon^2}{2} D^2 J_1(w_1)[e_1,e_1]+o(\varepsilon^2)\|e_1\|^2.\nonumber
	\end{align}

{Provided that $w_1>0$ is a solution to problem \eqref{P_scalar}, arguing according to \cite[Lemma 2.47]{RB1986}, we will prove that
 $w_1$ is a strict local minimum to $J_1$, and also that
there exists a positive constant $\gamma_1>0$ such that
	\begin{align}\label{Ineq:D2J1coercive}
		D^2J_1(w_1)[e_1,e_1]\geq 2\gamma_1\|e_1\|^2, 
	\end{align}
for all $ e_1 \in H_0^1(\Omega)$, in particular for $ \|e_1\| \leq1$. In fact, since
$-\Delta w_1-\mu_1w_1+w_1^{p-1}=0,$
and \(w_1>0\) in \(\Omega\), setting
\[
\cL_1:=-\Delta-\mu_1+(p-1)w_1^{p-2},
\]
it follows that
$\lambda_1(-\Delta-\mu_1+w_1^{p-2})=0.$
Moreover,
\[
(p-1)w_1^{p-2}
=w_1^{p-2}+(p-2)w_1^{p-2},
\]
with \((p-2)w_1^{p-2}>0\) in \(\Omega\). 
By the strict monotonicity of the principal
eigenvalue with respect to the potential,
we obtain
$\lambda_1(\cL_1)>0.$
Consequently, the quadratic form associated
with $\cL_1$ is coercive in $H_0^1(\Omega)$.
Thus, there exists $\gamma_1>0$ such that \eqref{Ineq:D2J1coercive} holds true,
for every $e_1\in H_0^1(\Omega)$.
In particular, $w_1$ is a strict local
minimum of $J_1$.}

This implies that
	\begin{align}\label{Ineq:I}
		J_1(w_1+ \varepsilon e_1)\geq J_1(w_1)+\left[{\varepsilon^2}\gamma_1 + {o(\eps^2)}\right]\|e_1\|^2 = J_1(w_1)+{\varepsilon^2}\left[\gamma_1 + {o(1)}\right]\|e_1\|^2.\end{align}

Now, we study the term in \eqref{QUAD}, with respect to the value of $p$. If, $p\in (2,4)$, we observe that for each fixed $\eps$, the functional $I_{\eps,\beta}$	is bounded from below, in the unitary ball, in view of the Poincare's inequality. 
    {In fact, fixing $e_1$ in the unitary ball, and taking $ e_2$ satisfying $\|e_1\|^2 + \|e_2\|^2 = 1$, we have
	\begin{align*}
		I_{\eps,\beta}(e_2)\geq \frac{\varepsilon^2}{2}\dint_{\Omega}\left(1 - \tfrac{\mu_2}{\lambda_1(\Omega)}\right)|\nabla e_2|^2+\frac{2|\beta|\varepsilon^{\frac{p}{2}}}{p}\dint_{\Omega_\eps}w_1^\frac{p}{2} |e_2|^\frac{p}{2}
		\geq \frac{\varepsilon^{2}}{2}\left(1 - \tfrac{\mu_2}{\lambda_1(\Omega)}\right)(1-\|e_1\|^2).
	\end{align*} Thus, we claim that,
for each fixed $\beta<0$ and $\gamma_2 \in(0,1/2)$, there exists
$\eps_\star(\beta)>0$ such that
\begin{align}\label{Ieps}
I_{\eps,\beta}(e_2)
\geq {\eps^2}\gamma_2\|e_2\|^2,
\end{align}
for every $\eps\in(0,\eps_\star)$ and
every $(e_1,e_2)\in\cH$ satisfying
$\|e_1\|^2+\|e_2\|^2=1$.
Indeed, suppose by contradiction that
this assertion is false. Then, fixing $\gamma_2 \in(0,1/2)$, there exist
$\eps_n\to0$ and $(e_{1,n},e_{2,n})\in\cH$,
with
$\|e_{1,n}\|^2+\|e_{2,n}\|^2=1,$ and $
\ e_{2,n}\neq0,$
such that
\[
I_{\eps_n,\beta}(e_{2,n})
<{\eps_n^2}\gamma_2\|e_{2,n}\|^2.
\]
Set
$\tilde e_{2,n}
:=\frac{e_{2,n}}{\|e_{2,n}\|}.$
Since $\|\tilde e_{2,n}\|=1$, passing
to a subsequence if necessary, we have
\[
\tilde e_{2,n}\rightharpoonup\tilde e_2
\quad\text{in }H_0^1(\Omega)
\]
and, by the compactness of the Sobolev
embeddings,
$\tilde e_{2,n}\to\tilde e_2
\quad\text{in }L^2(\Omega)\cap L^p(\Omega).$
Using the definition of $I_{\eps,\beta}$,
we obtain
\begin{align}\label{Iepscontr}
&\frac12\left(
1-\mu_2\|\tilde e_{2,n}\|_2^2
\right)+\frac{2|\beta|}{p}
\eps_n^{\frac p2-2}
\|e_{2,n}\|^{\frac p2-2}
\int_{\Omega_{\eps_n}}
w_1^{\frac p2}
|\tilde e_{2,n}|^{\frac p2}
<\gamma_2.
\end{align}}

{On the other hand, by the definition of
$\Omega_{\eps_n}$ and Holder's inequality, we get
\begin{align*}
&\int_{\Omega\setminus\Omega_{\eps_n}}
w_1^{\frac p2}
|\tilde e_{2,n}|^{\frac p2}\leq
\eps_n^{\frac p2}
\int_\Omega
|e_{1,n}|^{\frac p2}
|\tilde e_{2,n}|^{\frac p2}\leq C\eps_n^{\frac p2}
\to 0.
\end{align*}
Hence, the strong convergence in
$L^p(\Omega)$ implies that
\[
\int_{\Omega_{\eps_n}}
w_1^{\frac p2}
|\tilde e_{2,n}|^{\frac p2}
\longrightarrow
\int_\Omega
w_1^{\frac p2}
|\tilde e_2|^{\frac p2}.
\]
If $\tilde e_2\neq0$, since $w_1>0$
in $\Omega$, we have
$\int_\Omega
w_1^{\frac p2}
|\tilde e_2|^{\frac p2}>0.$
Moreover, since $2<p<4$ and
$0<\|e_{2,n}\|\leq1$, it yields
\[
\eps_n^{\frac p2-2}
\|e_{2,n}\|^{\frac p2-2}
\geq\eps_n^{\frac p2-2}
\to +\infty.
\]
This contradicts \eqref{Iepscontr}.
If $\tilde e_2=0$, the compactness
of the embedding
$H_0^1(\Omega)\hookrightarrow L^2(\Omega)$
implies that
$\|\tilde e_{2,n}\|_2\to0.$
Since the integral in \eqref{Iepscontr}
is nonnegative, we obtain
$\frac12\leq\gamma_2,$
which is again a contradiction if we take $\gamma_2 \in(0,1/2)$.
Therefore, \eqref{Ieps} holds.}

  {  Since  \(p/2-1>0\), taking $ \eps_\star$ smaller if necessary,  we may consider
$$
C|\beta|\varepsilon^{\frac p2-1}\leq\frac12 \min\left\{\frac{\gamma_1}{2},\gamma_2\right\}=: 
C_\star>0.$$ Therefore, combining \eqref{eq:J+beta2}, (\ref{Ineq:I}), and (\ref{Ieps}), it follows that	
\begin{align*}
		J^+_\beta(w_1+\varepsilon e_1,\varepsilon e_2)&\geq J_1(w_1)+\varepsilon^{2}\left(\frac{\gamma_1}{2} \|e_1\|^2+    {\gamma_2}\|e_2\|^2 -C|\beta|\varepsilon^{\frac p2-1}
        \right)   
        \\& \geq J_1(w_1)+\varepsilon^{{2}} C_\star(\|e_1\|^2+\|e_2\|^{2})  
        \\&=J_\beta(w_1,0)+\varepsilon^{2} C_\star> c_1,
	\end{align*}
	for every $\eps \in (0,\eps_\star).$ }

{On the other hand, suppose now that $p=4$. Recalling \eqref{QUAD}, we observe that
	\begin{align*}
	I_{\eps,\beta}(e_2)
=	\frac{\varepsilon^2}{2}\left[\dint_{\Omega}( |\nabla e_2|^2  -\mu_2e_2^2)+{|\beta|}\dint_{\Omega_\eps}w_1^2 |e_2|^2\right].
		\end{align*}
We aim to prove that the quadratic form on the right-hand side	is positive definite with respect to $e_2$.
We consider the operator
\begin{align}\label{positoper}
\cL_\beta
:=-\Delta+|\beta|\chi_{\Omega_\eps}w_1^2-\mu_2.
\end{align}
Since $w_1>0$ in $\Omega$, the dominated
convergence theorem yields
\[
\int_\Omega
(\mu_2-Mw_1^2)_+^{3/2}\to0
\quad\text{as }M\to+\infty.
\]
Thus, since $N=3,$ in view of Theorem \ref{t:LiYau},
we may choose $\beta_\star>0$ such that
\[
C(3)\int_\Omega
(\mu_2-\beta_\star w_1^2)_+^{3/2}<1.
\]
Consequently, the operator
$-\Delta+\beta_\star w_1^2-\mu_2$
has no non-positive eigenvalues.
Hence, there exists $\gamma_2>0$
such that
\[
\int_\Omega
\left[
|\nabla e_2|^2+
(\beta_\star w_1^2-\mu_2)e_2^2
\right]
\geq\gamma_2\|e_2\|^2.
\] Furthermore,  using the definition of
$\Omega_\eps$ and the Sobolev embedding
$H_0^1(\Omega)\hookrightarrow L^4(\Omega)$,
we obtain
\[
\begin{aligned}
\int_{\Omega\setminus\Omega_\eps}w_1^2e_2^2
&\leq\eps^2\int_\Omega e_1^2e_2^2\leq C\eps^2\|e_1\|^2\|e_2\|^2\leq C\eps^2\|e_2\|^2.
\end{aligned}
\]
Therefore,
\[
\begin{aligned}
I_{\eps,-\beta_\star}(e_2)
&=\frac{\eps^2}{2}
\left[
\int_\Omega
\left(
|\nabla e_2|^2+
(\beta_\star w_1^2-\mu_2)e_2^2
\right)-\beta_\star
\int_{\Omega\setminus\Omega_\eps}
w_1^2e_2^2
\right]\geq\frac{\eps^2}{2}
(\gamma_2-C\beta_\star\eps^2)
\|e_2\|^2.
\end{aligned}
\]}

{Taking $\eps_\star$ sufficiently small,
we obtain
\begin{align}\label{Ineq:II}
I_{\eps,-\beta_\star}(e_2)
\geq\frac{\gamma_2}{4}\eps^2\|e_2\|^2.
\end{align}
Combining \eqref{eq:J+beta2},
\eqref{Ineq:I}, and \eqref{Ineq:II},
with $\beta=-\beta_\star$, and reducing
$\eps_\star$ if necessary, we obtain
\[
J_{-\beta_\star}^+
(w_1+\eps e_1,\eps e_2)
\geq c_1+C_\star\eps^2
\]
for some $C_\star>0$ independent of
$\eps,e_1,e_2$.}

{Finally, for every $\beta<-\beta_\star$,
we have
\[
J_\beta^+(u,v)
\geq J_{-\beta_\star}^+(u,v).
\]
Consequently,
\[
J_\beta^+(w_1+\eps e_1,\eps e_2) \geq
J_\beta(w_1,0)+\varepsilon^2 C_\star\geq c_1+C_\star\eps^2,
\]
for every $\eps\in(0,\eps_\star)$
and every $\beta<-\beta_\star$.}

{
The same argument applies to the semitrivial
solution $(0,w_2)$. In the case $p=4$,
we choose $\beta_\star>0$ sufficiently
large so that both operators
\[
-\Delta+\beta_\star w_1^2-\mu_2
\quad\text{and}\quad
-\Delta+\beta_\star w_2^2-\mu_1
\]
have positive principal eigenvalues.
Taking the minimum of the corresponding
positive constants and radii, we obtain
the same $\eps_\star$ and $C_\star$
for both estimates. This completes
the proof.} 
\end{proof}

We are going to apply the classical Mountain Pass Theorem \cite{RB1986}, in the following setting. Either for $\beta < 0 $, if $2< p <\min\{4,2^*\}$, or for $\ \beta<- \beta_\star$, if $p=4$, we  consider the functional $J_\beta^+$, the sets of paths
\[\Gamma_{\beta}=\{\gamma\in C([0,1],\cH) :\gamma(0)=(w_1,0),\,\gamma(1)=(0,w_2)\};\]
and the min-max level \[{c_{\beta}=\inf_{\gamma\in \Gamma_{\beta}}\sup_{t\in[0,1]}J^+_\beta(\gamma(t)).}\]
In view of Proposition \ref{p:mountainpasslevel},  the level $c_{\beta}$ is well-defined and satisfies 
\begin{align}\label{eq:cbetalowerbound}	c_\beta\geq c_1 + \eps^2C_\star > c_1 \geq c_2 \geq m_\beta,
\end{align}
where \(0<\eps<\varepsilon_\star\). Therefore, we are able to prove the existence of the vectorial solution  in the Theorem \ref{t:mountainpass}. The next two results are essential for achieving this purpose.

\begin{theorem}\label{t:mpexistence}
Under the assumptions of Proposition \ref{p:mountainpasslevel}, there exists  a non-negative vectorial solution $(u_\beta,v_\beta)$ to problem \eqref{LS}, and a positive constant $\delta_\star(\beta)$ such that $$c_1+\delta_\star<J_\beta(u_\beta,v_\beta)<-\delta_\star.$$ 
\end{theorem}
\begin{proof}
	Under our assumptions, for each $\beta <0$ fixed, we are going to apply the mountain pass theorem to the functional $J_\beta^+$. Consider the path 
	$
	\gamma(s) =    
	\big((1-s)w_1, sw_2\big)
	$, for  all $s\in [0,1].$
	Note that $\gamma(s)\in \cL^-$, for each $s\in [0,1]$. Hence, there exists $t_{\gamma(s),\beta}>0$ given by Lemma \ref{projection}, such that, in view of \eqref{eq:Jbetanehari},  $\tilde \gamma (s):= t_{\gamma(s),\beta}\gamma(s)\in \cN_\beta$ satisfies
	\begin{align*}
		J_\beta^+\big(\tilde \gamma(s)\big) &= J_\beta\big(t_{\gamma(s),\beta} \big((1-s)w_1, sw_2\big)\big)\\ &=- \dfrac{p -2}{2p}\dfrac{\left(-\dint_{\Omega}|\nabla(1-s)w_1|^2-\mu_1| (1-s)w_1|^2+|\nabla sw_2|^2-\mu_2 |sw_2|^2\right)^\frac{p}{p-2}}{\left(\dint_{\Omega} |(1-s)w_1|^p+|sw_2|^p-2\beta |(1-s)w_1|^{\frac{p}{2}}|sw_2|^{\frac{p}{2}}\right)^{\frac{2}{p-2}}}.
	\end{align*}
    
	Observe that the map $s\mapsto J_\beta^+(\tilde \gamma(s))$ is continuous on the compact interval $[0,1]$. Since it is strictly negative for every $s\in[0,1]$, its maximum is attained and remains strictly negative.  Since $\tilde \gamma(s) \in \Gamma_\beta$, it follows that $c_\beta <0$. By standard arguments  as in \cite{RB1986}, we prove that $c_{\beta}$ is a critical value and admits a critical point $(u_\beta,v_\beta)\neq (0,0)$, such that $u_\beta \geq 0$ and $v_\beta\geq0$, since we are working with $J_\beta^+$.
    
	Notice that, in view of Proposition \ref{p:mountainpasslevel}, the mountain pass geometry yields that $c_{\beta}> c_1+\delta_\star(\beta)$, for every fixed \(\beta<0\) when \(2<p<4\), and for every \(\beta<-\beta_\star\) when \(p=4\), and for $\varepsilon=\varepsilon_\star(\beta)$, where we may set $\delta_\star : = \frac12 \varepsilon_\star^{2} C_\star$, and then we also have $c_\beta < -\delta_\star$, upon reducing $\varepsilon_\star$, if necessary.
	Hence, the solution $(u_\beta,v_\beta)$ cannot be semi-trivial. Indeed, by the uniqueness of $w_i$, the positive solution of the scalar problem \eqref{P_scalar}   for $i=1,2$, if $(u_\beta,v_\beta)$ were semi-trivial, then necessarily $(u_\beta, v_\beta)=(w_1, 0)$ or $(u_\beta,v_\beta)=(0,w_2)$. In such cases,  $J^+_\beta(u_\beta,v_\beta)=c_1$ or $J^+_\beta(u_\beta,v_\beta)=c_2$, respectively, leading to a contradiction, since $J_\beta(u_\beta, v_\beta) = c_\beta > c_1\geq c_2.$
	Therefore, $u_\beta \gneq 0$ and $v_\beta \gneq 0.$
\end{proof}

For the classical case $N=3$ and $p=4$, in view of the proof of Proposition \ref{p:mountainpasslevel},  if $\beta< -\beta_\star$ then the operator $\cL_\beta$ defined in \eqref{positoper} is positive definite. On the other hand,  since $\mu_2 > \lambda_1(\Omega)$,  we know that the operator $\cL_\beta$ is indefinite, for small $|\beta|$. Hence, we may assume  that for $\beta \in [-\beta_\star,0]$ there exists a finite number $n(0)$ of  non-positive eigenvalues, see Theorem \ref{t:LiYau}. This suggests that it might be possible to construct a linking structure, in order to find a vectorial solution of problem \eqref{LS}, when $\beta \in [-\beta_\star,0]$.

We now prove the segregation of the mountain-pass solution $(u_\beta,v_\beta)$ for the classical problem  $N= 3$ and $p=4$.

\begin{theorem}\label{t:segregation}
	Under the hypotheses of Proposition \ref{p:mountainpasslevel}, for $N=3$ and $p=4$, let   $(\beta_n)$ be a sequence of negative numbers, satisfying $\beta_n \to - \infty$. The sequence of non-negative solutions   $(u_{\beta_n}, v_{\beta_n})$, given by Theorem \ref{t:mpexistence}  satisfies, up to a subsequence,
	\begin{itemize}
		\item[(i)]  $(u_{\beta_n}, v_{\beta_n}) \to (u_\infty,v_\infty)$ in $\mathcal{C}^{0,\alpha}(\Omega)\cap H_0^1(\Omega)$, for all $\alpha \in (0,1);$
		\item[(ii)] $u_\infty \cdot v_\infty \equiv 0$ in $\Omega$ and $\int_\Omega \beta_n u_{\beta_n}^{2} v_{\beta_n}^{2} \to 0$, as $\beta_n \to -\infty;$
		\item[(iii)] the limit function $(u_\infty, v_\infty)$ satisfies the uncoupled system 
		\begin{equation}\label{system2}
			\left\{
			\begin{array}{lr}
				-\Delta u =\mu_1u- u^{3}\mbox{ \rm in } \Omega_u: = \{x\in\Omega \, : \,u>0\},\\
				-\Delta v =\mu_2 v- v^{3} \mbox{ \rm in } \Omega_v: = \{x\in\Omega \, : \,v>0\},\\
				u,v\in H_0^1(\Omega) \quad \mbox{and } \  \bar \Omega = \bar{\Omega_u\cup \Omega_v}.
			\end{array}
			\right.
		\end{equation}
	\end{itemize}
	Furthermore, if $u_\infty$ and $v_\infty$ are both nontrivial, then $$\mu_1>  \lambda_1(\Omega_u), \mbox{ and }  \mu_2 > \max\{ \lambda_1(\Omega_u), \lambda_1(\Omega_v)\} \geq \lambda_2(\Omega).$$
\end{theorem}
\begin{proof}	
	Using Lemma \ref{l:NehairBoundedness}, there exists a constant $C_0$ such that for any $\beta<0$ we get $\|(u_\beta,v_\beta)\|\leq C_0$. Then, there exists a $(u_\infty,v_\infty)\in H_0^1(\Omega)\times H_0^1(\Omega)$ such that
	\begin{align}(u_\beta,v_\beta)\rightharpoonup&(u_\infty,v_\infty)\mbox{ in }H_0^1(\Omega)\times H_0^1(\Omega);\nonumber\\
   ( u_\beta,v_\beta)\to&(u_\infty,v_\infty)\mbox{ in }L^r(\Omega)\times L^r(\Omega)\mbox{ for any }r\in[1,2^*);\nonumber\\
    (u_\beta(x),v_\beta(x))\to&(u_\infty(x),v_\infty(x)), \ x \mbox{ a.e. in }\Omega.\nonumber
	\end{align}
	Notice that $(u_\beta,v_\beta)$ satisfy that
	\begin{equation}
		\left\{
		\begin{array}{lr}
			-\Delta u \leq\mu_1u\mbox{ in }\Omega,\nonumber\\
			-\Delta v \leq\mu_2 v\mbox{ in }\Omega,\nonumber\\
			0\leq u ,v\in H_0^1(\Omega).\nonumber
		\end{array}
		\right.
	\end{equation}
	Applying \cite[Theorem 8.16]{GilbargTrudinger2001} for each  equation above with $f = 0$ and $g_1 = \mu_1u$ and $g_2 = \mu_2 v$, respectively, there exists a constant $\kappa>0$, such that
	\[  |u_{\beta_n}|_\infty \leq \kappa \|g_1\|_{\tfrac{q}{2}} ,|v_{\beta_n}|_\infty\leq \kappa \|g_2\|_{\tfrac{q}{2}}, \mbox{ with } \tfrac{q}{2} \in (\tfrac{N}{2},  2^*).\]
	Notice that, for $N= 3$ we have  $(\tfrac{N}{2},  2^*)$ is nonempty. Due to the Sobolev embeddings, and the uniform bound $ \|(u_{\beta_n},v_{\beta_n})\|\leq C_0$, given by Lemma \ref{l:NehairBoundedness}, we get a  uniform constant $C$, independent of ${\beta_n}$, such that
	\begin{align}
		|u_{\beta_n}|_\infty,|v_{\beta_n}|_\infty\leq C,\nonumber
	\end{align}
	for all $n \in \N$.
	In view of \cite[Theorems A, B, C]{STTZ} we derive $(i)-(iii)$, for the classical case $N=3$ and $p=4$, see also \cite[Theorem 1.1 and 1.2]{NorisTavarseTerraciniVerzini2010}.
Note that, items $(i)-(ii)$  yield
\begin{align}
	\lim_{n\to\infty}J_{\beta_n}(u_{\beta_n},v_{\beta_n})=J_1(u_\infty)+J_2(v_\infty).\nonumber
\end{align}
Since, from Theorem \ref{t:mpexistence} we get $J_{\beta_n}(u_{\beta_n},v_{\beta_n})\in(c_1,0)$, we know that
$J_1(u_\infty)+J_2(v_\infty)\in[c_1,0]$.
Furthermore, from \cite[Theorem C]{STTZ} we conclude that $u_\infty$ and $v_\infty$ cannot be zero simultaneously in a subdomain of $\Omega$, since   ${\{x\in\Omega \, : \, u_\infty=0\}}\cap\{x\in\Omega \, : \,v_\infty=0\}$ has zero measure. Thus, 	
it implies that ${\overline \Omega = \overline{\{x\in\Omega \, : \, u_\infty>0\}} \cup \overline{\{x\in\Omega \, : \,v_\infty>0\}}}$, and also
$(u_\infty,v_\infty)\equiv (0,0)$ cannot happen. 
On the other hand, if $\Omega_u = \{x\in\Omega \, : \,u_\infty>0\}=\emptyset$, in view of \cite[Theorems A, B, C]{STTZ} we get that $\{x\in\Omega \, : \,v_\infty =0\}$ has zero measure, and, $v_\infty \in \mathcal{C}^{0,\alpha}(\Omega)$ is, in fact, a classical solution of the scalar problem in $\Omega$. Thus, by the Maximum Principle,   $\Omega_v = \{x\in\Omega \, : \,v_\infty>0\} = \Omega$, and we conclude that $(u_\infty,v_\infty)=(0,w_2),$ by the uniqueness of a positive solution given by Lemma \ref{eq:Pi}. In this case, 
$ c_2 = J_2(v_\infty) = J_1(u_\infty)+J_2(v_\infty) \geq c_1 \geq c_2,$ leading again to 
a contradiction, unless $c_1=c_2$. Therefore, in case $c_1>c_2$, $(u_\infty,v_\infty) \neq (0,w_2)$, which implies that either $(u_\infty, v_\infty)$ is vectorial, or $(u_\infty, v_\infty) = (w_1, 0)$. 
In addition, if $c_1=c_2$, then the semi-trivial solution $(u_\infty, v_\infty) = (w_1,0)$ or $(u_\infty, v_\infty)= (0,w_2) = (0,w_1)$ can also happen.

Finally, 
if each one of $u_\infty$ and $v_\infty$ is nontrivial, then $u_\infty$ is the unique positive solution to the problem 
	\begin{equation}
	\left\{
	\begin{array}{lr}
		-\Delta u =\mu_1u- u^{3}\mbox{ \rm in } \Omega_u,\\
		u\in H_0^1(\Omega),
	\end{array}
	\right.
\end{equation}
which gives
\[
\lambda_1(\Omega_{u_\infty}) \le 
\frac{\int_{\Omega_{u_\infty}} |\nabla u_\infty|^2}
{\int_{\Omega_{u_\infty}} u_\infty^2}
= \mu_1 - \frac{\int_{\Omega_{u_\infty}} |u_\infty|^4}{\int_{\Omega_{u_\infty}} u_\infty^2}
< \mu_1\leq\mu_2.
\]
Similarly, we get $\lambda_1(\Omega_{v_\infty}) < \mu_2 $. 

Henceforth, we recall the following optimal partition  variational characterization of the second eigenvalue of the Laplacian operator in $\Omega$ with Dirichlet boundary condition, which can be deduced by \cite[Theorem 1.17]{HelfferHoffmannOstenhofTerracini2009}, for $k=2$, namely,
\begin{equation}\label{eq:lambda2_partition}
\lambda_2(\Omega)
=
\inf_{\substack{A,B\subset\Omega\\ A\cap B=\emptyset, \
\overline{A\cup B}=\overline{\Omega}}}
\max\{\lambda_1(A),\lambda_1(B)\},
\end{equation}
whose optimal partition is given by the two nodal domains
of an eigenfunction associated with $\lambda_2(\Omega)$. 
Since \(\Omega_u\) and \(\Omega_v\) are disjoint positivity domains, the variational characterization of the second Dirichlet eigenvalue gives
$\lambda_2(\Omega)
\leq
\max\{\lambda_1(\Omega_u),\lambda_1(\Omega_v)\}$.
Thus, we conclude that $\mu_2 > \max\{ \lambda_1(\Omega_u), \lambda_1(\Omega_v)\} \geq \lambda_2(\Omega)$, completing the proof.
\end{proof}
\begin{remark}  According to the authors of \cite{NorisTavarseTerraciniVerzini2010}, Theorems 1.1 and 1.2 in \cite{NorisTavarseTerraciniVerzini2010} remain valid in higher dimensions, provided that the cubic nonlinearities are replaced by subcritical ones. In this setting, the segregation result in Theorem \ref{t:segregation} could likewise be extended to dimensions  $N=2$, $N=4$ and $N=5$, since the interval  $(\tfrac{N}{2},  2^*)$ is nonempty for these dimensions.
\end{remark}
\begin{proof}[Proof of Theorem \ref{t:mountainpass}]
Combining the results proved in Theorems \ref{t:mpexistence} and \ref{t:segregation}, we obtain \eqref{MPenergy}, and conclude that, in case the limit vector $(u_\infty,v_\infty)$  is a semi-trivial solution,  it coincides with $(w_1,0)$ or $(0,w_2) = (0,w_1)$. Otherwise, it is a vectorial weak solution to system \eqref{system2},  with $u_\infty v_\infty \equiv 0$ in $\Omega$ and $\bar\Omega = \bar{\Omega_u\cup\Omega_v}$. Setting $(u_\infty, v_\infty) = (w_\infty^+,w_\infty^-)$, this means that $w_\infty = w_\infty^+ - w_\infty^-$ is a sign-changing weak solution  to the scalar problem \eqref{e:SignChanging}. Indeed, using the weak formulation of \eqref{system2}, subtracting the second equation from the first one, we obtain the following integral identity
\begin{align*} \int_{\Omega} \nabla w_\infty \cdot \nabla \varphi  dx &= \int_{\Omega_u} \nabla w_\infty^+ \cdot \nabla \varphi  dx - \int_{\Omega_v} \nabla w_\infty^- \cdot \nabla \varphi  dx \\
& = \mu_1 \int_{\Omega_u} w_\infty^+ \varphi  dx - \int_{\Omega_u} (w_\infty^+)^3 \varphi  dx - \mu_2 \int_{\Omega_v} w_\infty^- \varphi  dx + \int_{\Omega_v} (w_\infty^-)^3 \varphi  dx \\
& = \int_{\Omega} (\mu_1w_\infty^+ - \mu_2w_\infty^-) \varphi  dx - \int_{\Omega} w_\infty^3 \varphi  dx, \end{align*}
 for every test function $\varphi \in H_0^1(\Omega)$. This is exactly the weak formulation of problem \eqref{e:SignChanging}.  
\end{proof}
\begin{remark} Recall that, in case $\mu_1 = \mu_2 = \mu$, it yields $c_1 = c_2$ and $w _1= w_2$. By Theorem \ref{t:segregation} and  \ref{t:mountainpass}, when $(u_\infty,v_\infty)$ is vectorial, we obtain
	$w_\infty = w_\infty^+ - w_\infty^-$, a sign-changing weak solution  to the scalar problem
\eqref{e:SignChanging}, and conclude that, necessarily, $\mu > \lambda_2(\Omega)$. This reaffirms the result found in \cite{CardosoFurtadoMaia2024}, where the authors assume   $\mu > \lambda_2(\Omega)$ in order  to obtain a sign-changing solution to \eqref{e:SignChanging}. 
\end{remark}

\section{Vectorial Solutions to the Cooperative Problem with \texorpdfstring{$\beta>1$}{beta>1}}

We begin by establishing a non-existence result for problem \eqref{LS} when $\beta>1$. More precisely, we show that problem \eqref{LS} admits no positive vectorial solution, that is, there is no solution $(u,v)\in\cH$ such that $u>0$ and $v>0$ in $\Omega$.

\begin{proof}[Proof of Theorem \ref{t:nonexistencebeta>0}]
Using an approach similar to that developed in \cite[pp. 1865]{CZ2021}, we argue by contradiction. Assume that there exists a vectorial positive solution $(u,v)$ to problem \eqref{LS}. {We claim that $u\leq v$ in $\Omega$.
Suppose, by contradiction, that
\[
\Omega_+:=\{x\in\Omega:u(x)>v(x)\}
\neq\emptyset.
\] Multiplying the first equation by $v$ and the second one by $u$, integrating by parts on $\Omega^\varepsilon_+=\{x\in\Omega \, : \, u(x)>v(x)+\varepsilon\}$, for a fixed small $\varepsilon>0$, and subtracting the first equation from the second one,  we get}
	\begin{align}\label{eq:2-1}
\dint_{\partial\Omega^\varepsilon_+}\Big(v\frac{\partial u}{\partial n} &-u\frac{\partial v}{\partial n}\Big)dS = \dint_{\Omega^\varepsilon_+}\left[(\mu_2-\mu_1)+\left(u^{p-2}-v^{p-2}\right)
+\beta\left(v^{\frac{p}{2}-2}u^{\frac{p}{2}}-u^{\frac{p}{2}-2}v^{\frac{p}{2}}\right)\right]uvdx \nonumber\\
		&=(\mu_2-\mu_1)\dint_{\Omega^\varepsilon_+} uvdx +{\color{red}\dint_{\Omega_+^\varepsilon}\left[\big(u^{p-2}-v^{p-2}\big)
+\beta u^{\frac{p}{2}-2}v^{\frac{p}{2}-2}(u^2-v^2)\right]uvdx},
	\end{align}
{for every $\beta>1$.
Set $w=u-v$, since $u=v+\varepsilon$ on
$\partial\Omega_+^\varepsilon$, we obtain
\[
v\frac{\partial u}{\partial n}
-u\frac{\partial v}{\partial n}
=
v\frac{\partial w}{\partial n}
-\varepsilon\frac{\partial v}{\partial n}.
\]
For every sufficiently small regular value \(\varepsilon>0\) of \(w\), the outward normal derivative satisfies  
\[
\frac{\partial w}{\partial n}
=-|\nabla w|\leq0
\quad\text{on }
\partial\Omega_+^\varepsilon.
\]
Consequently, by \eqref{eq:2-1},
\begin{align}\label{eq:boundary-estimate}
\varepsilon
\int_{\partial\Omega_+^\varepsilon}
|\nabla v|dS &\geq
\int_{\partial\Omega_+^\varepsilon}
\left(
v\frac{\partial w}{\partial n}
-\varepsilon\frac{\partial v}{\partial n}
\right)dS  \geq(\mu_2-\mu_1)\dint_{\Omega^\varepsilon_+} uvdx \notag \\
& \quad +\dint_{\Omega_+^\varepsilon}\big(u^{p-2}-v^{p-2}\big)uvdx + \dint_{\Omega_+^\varepsilon}\beta u^{\frac p2-1}v^{\frac p2-1}(u^2-v^2)dx.
\end{align}
}

{On the other hand,
for every sufficiently small $\delta>0$,
\[
\begin{aligned}
&\int_0^\delta
\left(
\int_{\partial\Omega_+^t}
|\nabla v|\,dS
\right)dt=
\int_{\{0<w<\delta\}}
|\nabla v||\nabla w|\,dx
<+\infty.
\end{aligned}
\]
Hence, there exists a sequence
of regular values $\eps_n\to0^+$ of $w$
such that
\[
\eps_n
\int_{\partial\Omega_+^{\eps_n}}
|\nabla v|\,dS\to0.
\]Therefore, 
passing to the limit in
\eqref{eq:boundary-estimate}, we obtain
\[ 0
 \geq (\mu_2-\mu_1)\dint_{\Omega_+} uvdx + \dint_{\Omega_+}\big(u^{p-2}-v^{p-2}\big)uvdx + \dint_{\Omega_+}\beta u^{\frac p2-1}v^{\frac p2-1}(u^2-v^2)dx
> 0,\]
since \(\mu_2\geq\mu_1\), \(p>2\),  \(\beta>1\), and $ u>v$ in $\Omega_+$. This is a contradiction; hence, $\Omega_+=\emptyset$, which
implies that $v\geq u>0$ in $\Omega$.}

	This implies that $\Omega^0_+= \{x\in\Omega \,:\, u(x)\geq v(x)\} = \{x\in\Omega \,:\, u(x)=v(x)\}$, and then $\{x\in\Omega \,:\, u(x)>v(x)\}=\emptyset$. Thereby, $v\geq u>0$ in $\Omega$. In addition, multiplying the first equation by $\phi_1$, the principal eigenfunction of $-\Delta$, and integrating over $\Omega$, we get
	\begin{align}
		\lambda_1(\Omega)\dint_\Omega u\phi_1 dx&= \mu_1\dint_\Omega u\phi_1 dx- \dint_\Omega u^{p-1}\phi_1 dx+\beta\dint_\Omega {\color{red}u^{\frac{p}{2}-1}v^{\frac{p}{2}}\phi_1} dx \\ &\geq\mu_1\dint_\Omega u\phi_1 dx+(\beta-1)\dint_\Omega u^{p-1}\phi_1 dx.\nonumber
	\end{align}
	However, this is a contradiction, since $ \beta>1 $ and $\mu_1> \lambda_1(\Omega)$. Therefore, problem \eqref{LS} does not admit a positive vectorial solution.
\end{proof}

In what follows, we establish the existence of a positive-energy solution $(\what u_\beta, \what v_\beta)$ which, by Lemma \(\ref{eq:Pi}\), must be vectorial. Furthermore, by Theorem \(\ref{t:nonexistencebeta>0}\), such a solution cannot have both components positive. Namely, it is either a nodal or a semi-nodal solution.  To obtain such a solution, we apply a linking argument. 

Let us define $Q: \cH \to \r$, given by
\begin{align}
	Q(u,v):=\dint_{\Omega}|\nabla u|^2-\mu_1u^2+\dint_{\Omega}|\nabla v|^2-\mu_2 v^2,\nonumber
\end{align}
and decompose the space $\cH:=H_0^1(\Omega)\times H_0^1(\Omega)$ accordingly 
\begin{align}
	\cH= \cH^-\oplus \cH^0\oplus \cH^+ \mbox{ with }
	Q|_{\cH^+}>0, \mbox{} -Q|_{\cH^-}>0, \mbox{ and }Q|_{\cH^0}=0.\nonumber
\end{align}
We denote by $\{\phi_i\}$ an orthonormal basis for $H_0^1(\Omega)$ of eigenfunctions  associated with 
the sequence of eigenvalues $\lambda_i(\Omega)\to +\infty$ of Laplacian operator with Dirichlet boundary condition on $\Omega.$
	{Further, let $\kappa_1, \kappa_2 \in \N$ be the least indexes such that  $$\lambda_{\kappa_i+1}(\Omega) > \mu_i> \lambda_{1}(\Omega), \quad \kappa_2\geq \kappa_1\geq1,$$ and define $E_1:= {\rm span}\big\{
\phi_i \ :\  1\leq i \leq \kappa_1 
\big\} $ and $E_2:= {\rm span}\big\{
\phi_i\ :\ 1  \leq i\leq \kappa_2 
\big\}.$} Then, 	
	 $\text{codim }(\cH^+)= \dim(\cH^-\oplus \cH^0) = \dim(E_1\times E_2) = \kappa_1 + \kappa_2$, since
	 $\{(\phi_i,0), (0, \phi_i) : i \in  \N \}$ is an orthonormal basis for $\cH$, and $Q(0,\phi_j), Q(\phi_i,0) >0$, for all $i > \kappa_1$, $j> \kappa_2$. 

{ Since each eigenspace associated to $\lambda_i(\Omega)$ is finite dimensional, by definition, $E_i$ is a finite dimensional space, for $i=1,2$. 
We use the following auxiliary result to construct our linking structure.
\begin{lemma}\label{lem:disjoint-projections} Let $P_{E_i}: H_0^1(\Omega) \to E_i$ be the orthogonal projection maps from $H^1_0(\Omega)$ onto $E_i$, for $i=1,2.$
Then, for every nonempty open subset $U_i\subset\Omega$,
there exist functions
$\psi^i_1,\ldots,\psi^i_{\kappa_i}\in C_c^\infty(U_i)$
with pairwise disjoint supports such that
$\{P_{E_i}\psi^i_1,\ldots,P_{E_i}\psi^i_{\kappa_i}\}$
is a basis of $E_i$.
\end{lemma}}

\begin{proof} {Since for each $i = 1,2$ the argument is independent, as a matter of simplicity, throughout this proof, let us denote $E_i = E$, $U_i = U$,  $P_{E_i} = P_E$, and $\kappa_i = k$.}

{Provided that the eigenfunctions of the Laplacian are real
analytic in $\Omega$, their restrictions to any
nonempty open subset $U\subset\Omega$ remain linearly
independent. Indeed, if we consider
$\sum_{j=1}^k a_j\phi_j=0$ in $U,$
by the analyticity and connectedness of $\Omega$, we conclude 
that this identity holds throughout $\Omega$.
The linear independence of the eigenfunctions then
gives $a_j=0$ for every $j$.
Consequently, there exist distinct points
$x_1,\ldots,x_k\in U$ such that for 
$A:=\big(\phi_i(x_j)\big)_{i,j=1}^k,$ we have $\det A\neq0$.
Otherwise, the vectors
$(\phi_1(x),\ldots,\phi_k(x))$, $x\in U$,
would fail to span $\mathbb R^k$, contradicting
the preceding linear independence.}

{Choosing a nonnegative mollifier function
$\eta\in C_c^\infty(B_1(0))$ satisfying
$\int_{\mathbb R^N}\eta=1$, and defining
$\psi_j^\varepsilon(x)
=\varepsilon^{-N}
\eta\left(\tfrac{x-x_j}{\varepsilon}\right),$
for sufficiently small $\varepsilon>0$,  these
functions belong to $C_c^\infty(U)$ and have
pairwise disjoint supports.
Moreover,
\[
\int_\Omega\psi_j^\varepsilon\phi_i
\longrightarrow\phi_i(x_j)
\qquad\text{as }\varepsilon\to0.
\]
Since the eigenfunctions are orthonormal in
$H_0^1(\Omega)$, we have
\[
P_E\psi_j^\varepsilon
=\sum_{i=1}^k
\lambda_i(\Omega)
\left(\int_\Omega\psi_j^\varepsilon\phi_i\right)
\phi_i \quad
\mbox{ where } \quad
M_\varepsilon
=\left(
\lambda_i(\Omega)\int_\Omega\psi_j^\varepsilon\phi_i
\right)_{i,j=1}^k
\]
is the coefficient matrix of the projection. Hence, as $\eps \to 0$, we obtain
\[
\det M_\varepsilon
\longrightarrow
\left(\prod_{i=1}^k\lambda_i(\Omega)\right)\det A
\neq0.
\]
Therefore for sufficiently small $\varepsilon$,
the functions
$P_E\psi_1^\varepsilon,\ldots,
P_E\psi_k^\varepsilon$
must form a basis of $E$.
Finally, normalizing each $\psi_j^\varepsilon$
in $H_0^1(\Omega)$, if necessary, preserves both
the disjointness of their supports and the
linear independence of their projections.}
\end{proof}

{Now, in view of Lemma~\ref{lem:disjoint-projections} we  may choose three nonempty open sets
$U_1,U_2,U_3\Subset\Omega$ with pairwise disjoint
closures, with $U_i$ corresponding to $E_i$, for $i=1,2$, respectively, and consider a set $$\{(u_i,0), (0, v_j) : 1\leq i \leq \kappa_1, \ 1  \leq j\leq \kappa_2  \}$$ of unitary functions with disjoint support, such that $$\{P_{E_1}(u_1),\ldots,P_{E_1}(u_{\kappa_1})\} \mbox{ 
and }
\{P_{E_2}(v_1),\ldots,P_{E_2}(v_{\kappa_2})\}$$
form  bases of $E_1$ and $E_2$, respectively. Then, we define
\[
\begin{aligned}
\cV_0:={\rm span}\big\{
&(u_i,u_i),(v_j,-v_j) \ :\  1\leq i \leq \kappa_1, \ 1  \leq j\leq \kappa_2 
\big\}.
\end{aligned}
\] Since the supports are disjoint,
$
\dim {\cV}_0=\kappa_1 + \kappa_2=  \dim(\cH^-\oplus \cH^0) = \text{codim }(\cH^+).
$
Moreover, since the supports of these functions
are contained in $U_1\cup U_2$, we may choose
a nonzero function $\psi\in C_c^\infty(U_3)$ such that
\begin{equation}\label{suppdisjoint}
\operatorname{supp}\psi\cap
\left(
\bigcup_{i=1}^{\kappa_1}\operatorname{supp}u_i
\cup
\bigcup_{j=1}^{\kappa_2}\operatorname{supp}v_j
\right)=\varnothing.
\end{equation}}

{ The following lemma will be useful to set the linking sets. 
 \begin{lemma}\label{lem:geo0} 
Let $\cM: = E_1\times E_2$,  $\cM^\perp = \cH^+$, and let $P_{\cM}$ denote the orthogonal projection from $\cH$ onto $\cM$. Then, $P_{\cM}(\cV_0) = \cM$ and consequently,   $\cV_0 + \cH^+ = \cH$.
\end{lemma}}
\begin{proof}
Let us define
$ F_1:=\operatorname{span}\{u_1,\ldots,u_{\kappa_1}\}$ and $
F_2:=\operatorname{span}\{v_1,\ldots,v_{\kappa_2}\}.$ 
By construction, the maps
$
P_{E_1}|_{F_1}:F_1\longrightarrow E_1$ and $
P_{E_2}|_{F_2}:F_2\longrightarrow E_2$
are linear isomorphisms.
We want to show that
$P_{\mathcal M}(\mathcal V_0)=E_1\times E_2,$ then given $(u,v)\in E_1\times E_2$, we are going to find $(u_0,v_0)\in \cV_0$, such that $P_{\cM}(u_0,v_0) = (u,v)$. Since $P_{E_1}|_{F_1}$ is surjective, we can choose $f_1\in F_1$ satisfying
${P_{E_1}f_1=\frac12(u+P_{E_1}v) \in E_1.}$
Moreover, because $P_{E_2}|_{F_2}$ is surjective, we can choose $f_2\in F_2$ satisfying
$ P_{E_2}f_2=P_{E_2}f_1-v \in E_2.$
Observe that
$(f_1+f_2,f_1-f_2)\in\mathcal V_0.$
On the other hand, since $E_1\subset E_2$, we have
$P_{E_1}P_{E_2}=P_{E_1}.$
Hence, 
\[
\begin{aligned}
P_{E_1}(f_1+f_2)
&=P_{E_1}f_1+P_{E_1}P_{E_2}f_2=P_{E_1}f_1+P_{E_1}(P_{E_2}f_1-v)=2P_{E_1}f_1-P_{E_1}v=u.
\end{aligned}
\]
In addition,  
$P_{E_2}(f_1-f_2)
=P_{E_2}f_1-P_{E_2}f_2=v.$
Thus,
$P_{\mathcal M}(f_1+f_2,f_1-f_2)=(u,v).$
Since $(u,v)\in\mathcal M$ is arbitrary,
 ${P_{\mathcal M}(\mathcal V_0)=\mathcal M,}$ and $P_{\mathcal M}|_{\mathcal V_0}$ is surjective. Provided that $\dim\mathcal V_0=\dim\mathcal M$, such a projection is also injective.
Therefore,
$ \mathcal H=\mathcal V_0\oplus\mathcal H^+.
$
\end{proof}

{With Lemma \ref{lem:geo0} in hands, using the  function $\psi$ chosen in \eqref{suppdisjoint}, we are able to define
$\mathcal V_1
:=\mathcal V_0\oplus
\operatorname{span}\{(\psi,\psi)\},$  which implies that $\dim\mathcal V_1
=\kappa_1+\kappa_2+1
=\operatorname{codim}\mathcal H^++1,$
and
$\mathcal V_1+\mathcal H^+=\mathcal H,$
because \(\mathcal V_0\subset\mathcal V_1\) and $\cH = \cV_0 + \cH^+$.}

Furthermore, given $(u,v) \in \cV_0$, there exists $a_i, b_j \in \r$, for $1\leq i \leq \kappa_1$, $1\leq j \leq \kappa_2$,  such that $$(u,v)= \sum_{i=1}^{\kappa_1}a_i(u_i,u_i) + \sum_{j=1}^{\kappa_2}b_j(v_j,-v_j) = \left(\sum_{i=1}^{\kappa_1}a_iu_i + \sum_{j=1}^{\kappa_2}b_jv_j,\sum_{i=1}^{\kappa_1}a_iu_i -\sum_{j=1}^{\kappa_2}b_jv_j\right).$$ Moreover, setting $f_1: = \sum_{i=1}^{\kappa_1}a_iu_i \in F_1$ and $f_2: =\sum_{j=1}^{\kappa_2}b_jv_j \in F_2$, we get $(u,v) = (f_1 + f_2, f_1- f_2)$ and,
$|u| = |f_1| + |f_2| = |v|$, almost everywhere in $\Omega.$ 
Similarly,  every \((u,v)\in\mathcal V_1\) can be written as
\[
(u,v)=(f_1+f_2+t\psi,f_1-f_2+t\psi),
\]
for \(f_i\in F_i\), for $i=1,2$ and \(t\in\mathbb R\). In view of \eqref{suppdisjoint}, we also have $
|u|=|v| \ \text{a.e. in }\Omega.$ 
In particular,
\[
\int_\Omega|\nabla u|^2
=\int_\Omega|\nabla v|^2.
\]
Thus, 
\begin{align}\label{eq:JV0}
J_\beta(u,v)&=\dint_{\Omega}|\nabla  u|^2-\frac{\mu_1+\mu_2}{2}\dint_{\Omega}|u|^2+\frac{2(1-\beta)}{p}\dint_{\Omega}|u|^p, \quad \mbox{for all }  \ (u,v) \in \cV_1.
	\end{align}

Under this setting, we obtain the following linking structure.

\begin{lemma}\label{lem:geo3}
    There exists a $R_0>0$ such that $J_\beta|_{\partial B_{R}(0)\cap \cV_1}<0$, for all $R > R_0$.
\end{lemma}
\begin{proof}
Notice that $\cV_1$ is a finite-dimensional space, so the norms $\|\cdot\| $ and $|\cdot|_{p}$ are equivalent under this setting. Since $\beta >1$, from \eqref{eq:JV0}, 	there exists a constant $C_0>0$ such that,  for every $(u,v) \in \cV_1$, it holds
	\begin{align}
		J_\beta( u,v)&\leq\dint_{\Omega}|\nabla u|^2+\frac{2(1-\beta)}{p}\dint_{\Omega}|u|^p\leq\|u\|^2-\frac{2C_0(\beta-1)}{p}\|u\|^p<0, 
        \nonumber
	\end{align}
    for $ \|(u,v)\|=\sqrt2\|u\|> \sqrt2 {\left(\frac{p}{{2C_0(\beta-1)}}\right)^{\frac{1}{p-2}}} =: R_0 $.
\end{proof}
\begin{lemma}\label{lem:geo2}
	Setting \begin{align}
		S_r:=\partial B_r(0)\cap \cH^+,\nonumber
	\end{align}  there exists a constant $\alpha>0$ such that $J_\beta|_{S_r}\geq\alpha$, for a sufficiently small $r>0.$ 
\end{lemma}
\begin{proof}
	First, notice that $Q(u,v)^{\frac{1}{2}}$ is an  equivalent norm to  $\|(u,v)\|$ in $\cH^+$. Then, there exists a constant $C>0$ such that, using Young inequality and the Sobolev embeddings, it holds
	\begin{align}
		J_\beta(u,v)&\geq C\|(u,v)\|^2+\frac{1}{p}\dint_{\Omega} |u|^p + |v|^p-\frac{2\beta}{p}\dint_{\Omega} |u|^{\frac{p}{2}}|v|^{\frac{p}{2}}\nonumber\\
		&\geq C\|(u,v)\|^2-\dfrac{1}{p}(\beta-1)(|u|_p^p+|v|_p^p)\nonumber\\
		&\geq  C\|(u,v)\|^2 - C'(\beta-1)(\|u\|^p+\|v\|^p)\nonumber\\
		&\geq C\|(u,v)\|^2 -C''(\beta-1)\|(u,v)\|^p \\
		&\geq \|(u,v)\|^2(C-C''(\beta-1)\|(u,v)\|^{p-2})  =\alpha_\beta >0,\nonumber
	\end{align}
	if $\|(u,v)\|^2=r_\beta^2 >0$, sufficiently small for each $\beta>1$ fixed.
\end{proof}

Enlarging $R_0$, if necessary, we may assume that $r_\beta<R_0$. 
Under the linking geometry, we are able to define the level
\begin{align}\label{c}
c:=\inf_{h\in\Gamma}\sup_{\cV_1}J_\beta(h(u,v))
\end{align}
where
\begin{align}
	\Gamma:=\{h\in \mathcal{C}^0(\cH,\cH)\, : \ h  \text{ is odd, and } h(u) = u \ \text{if } u \in \cV_1 \text{ and } \|u\| \geq R\}.\nonumber
\end{align}
In view of Lemma \ref{lem:geo2}, we know that 
$
c\geq\alpha>0.
$

In order to obtain a critical point at level $c,$ we prove that the functional $J_\beta$ satisfies the $(C)_c$ condition.

\begin{lemma}\label{lem:Cc}
	Under our assumptions, the functional $J_\beta$ satisfies the $(C)_c$ condition.
\end{lemma}
\begin{proof}
			Let $(u_n,v_n)\subset \mathcal H$ be a Cerami sequence for $J_\beta$ at
			level $c>0$, that is,
			\begin{equation}\label{cerami}
				J_\beta(u_n,v_n)\to c
				\quad\text{and}\quad
				\bigl(1+\|(u_n,v_n)\|\bigr)
				\|J_\beta'(u_n,v_n)\|_{\mathcal H^*}\to 0.
			\end{equation}
			In view of the compact Sobolev embeddings, it is enough to prove the boundedness of $(u_n, v_n)$,  in $\mathcal H$.			
			Suppose, by contradiction, that
		$
			\|(u_n,v_n)\|\to+\infty.
		$
			Set
			\[
			w_n=(\widetilde u_n,\widetilde v_n)
			:=\frac{(u_n,v_n)}{\|(u_n,v_n)\|}.
			\]
			Then
		$
				\|w_n\|=1.
		$
			Since $\mathcal H$ is reflexive, up to a subsequence,
			$
			w_n\rightharpoonup w=(\tilde u,\tilde v)
			\ \text{in }\mathcal H.
		$
			Moreover, since $\Omega$ is bounded and $2<p<2^*$, the embedding
		$
			H_0^1(\Omega)\hookrightarrow L^p(\Omega)
		$
			is compact. Hence
			\begin{equation}\label{strongLp}
				\widetilde u_n\to \tilde u,
				\qquad
				\widetilde v_n\to \tilde v
				\quad\text{strongly in }L^p(\Omega).
			\end{equation}	
			We now divide the equation for $J_\beta'(u_n,v_n)$ by
			$\|(u_n,v_n)\|^{p-1}$. From the Cerami sequence definition, we know that
			\[
		\|(u_n,v_n)\|	\|J_\beta'(u_n,v_n)\|_{\mathcal H^*}
			=o_n\left(1\right),
			\]
			and therefore, for every $(\varphi,\psi)\in\mathcal H$,
			\begin{align}
				o_n(1)=&
				\frac{J_\beta'(u_n,v_n)(\varphi,\psi)}
				{\|(u_n,v_n)\|^{p-1}}
				=
				\|(u_n,v_n)\|^{2-p}
				\int_\Omega
				\left(
				\nabla\widetilde u_n\cdot\nabla\varphi
				+\nabla\widetilde v_n\cdot\nabla\psi
				-\mu_1\widetilde u_n\varphi
				-\mu_2\widetilde v_n\psi
				\right)
				\nonumber\\
				&+
				\int_\Omega
				\left(
				|\widetilde u_n|^{p-2}\widetilde u_n\varphi
				+
				|\widetilde v_n|^{p-2}\widetilde v_n\psi
				\right)
				-\beta\int_\Omega
				\left(
				|\widetilde u_n|^{\frac p2-2}
				|\widetilde v_n|^{\frac p2}
				\widetilde u_n\varphi
				+
				|\widetilde u_n|^{\frac p2}
				|\widetilde v_n|^{\frac p2-2}
				\widetilde v_n\psi
				\right)
				+o_n(1).\nonumber\\
				\label{normalizedderivative}
			\end{align}
			Since $p>2$ and $\|(u_n,v_n)\|\to\infty$, the first integral in
			\eqref{normalizedderivative} converges to zero. Using
			\eqref{strongLp}, we may pass to the limit and obtain
			\begin{align}
				0={}&
				\int_\Omega
				\left(
				|\tilde u|^{p-2} \tilde u\varphi
				+
				|\tilde v|^{p-2}  \tilde v\psi
				\right)
				-\beta\int_\Omega
				\left(
				|\tilde u|^{\frac p2-2}|\tilde v|^{\frac p2}\tilde u\varphi
				+
				|\tilde u|^{\frac p2}|\tilde v|^{\frac p2-2}\tilde v\psi
				\right)
				\label{limitingequation}
			\end{align}
			for every $(\varphi,\psi)\in\mathcal H$.
			Taking $(\varphi,\psi)=(\tilde u,0)$ and $(\varphi,\psi)=(0,\tilde v)$ in
			\eqref{limitingequation}, respectively, gives
			\begin{align}
				\int_\Omega |\tilde u|^p
				&=
				\beta\int_\Omega |\tilde u|^{p/2}|\tilde v|^{p/2}=
				\int_\Omega |\tilde v|^p.
				\label{eq:v}
			\end{align}
			Furthermore, the limiting equation  \eqref{limitingequation} holds in the dual of $\cH$, hence the corresponding $L^{p'}(\Omega)$-functions vanish  almost everywhere in
			$\Omega$, namely
			\[
			|\tilde u|^{p-2}\tilde u
			=
			\beta |\tilde u|^{\frac p2-2}|\tilde v|^{\frac p2}\tilde u,
			\mbox{			and }
			|\tilde v|^{p-2}\tilde v
			=
			\beta |\tilde u|^{\frac p2}|\tilde v|^{\frac p2-2}\tilde v.
			\]
			Hence, on the set where $\tilde u \neq 0$ and $\tilde v\neq 0$,
			we obtain $$
			|\tilde u|^{p/2}=\beta |\tilde v|^{p/2},
			\
			|\tilde v|^{p/2}=\beta |\tilde u|^{p/2},
			\quad \mbox{ and consequently, } \quad 
			|\tilde u|^{p/2}=\beta^2|\tilde u|^{p/2}.
			$$
			Since $\beta>1$, it follows that
		$	\tilde u=\tilde v=0$
			almost everywhere in $\Omega.$
			Thus, it follows that
	$
				w_n\rightharpoonup0
				\ \text{in }\mathcal H.
		$	
			By the compact embedding
			$H_0^1(\Omega)\hookrightarrow L^2(\Omega)$, we also have
			$
			\widetilde u_n\to0,
			\
			\widetilde v_n\to0
			\ \text{strongly in }L^2(\Omega).
			$
			On the other hand, testing $J_\beta'(u_n,v_n)$ with $(u_n,v_n)$
			and using \eqref{cerami}, we have
			$
			J_\beta'(u_n,v_n)(u_n,v_n)=o(1).
			$
			Therefore,
			\begin{align}
				J_\beta(u_n,v_n)
				-\frac1pJ_\beta'(u_n,v_n)(u_n,v_n)
				={}&
				\frac{p-2}{2p}
				\int_\Omega
				\left(
				|\nabla u_n|^2-\mu_1u_n^2
				+
				|\nabla v_n|^2-\mu_2v_n^2
				\right).
			\end{align}
			Since $J_\beta(u_n,v_n)\to c$, it follows that
			\begin{equation}\label{quadraticbounded}
				\int_\Omega
				\left(
				|\nabla u_n|^2-\mu_1u_n^2
				+
				|\nabla v_n|^2-\mu_2v_n^2
				\right)
				=
				\frac{2p}{p-2}c+o_n(1).
			\end{equation}
			Dividing \eqref{quadraticbounded} by
			$\|(u_n,v_n)\|^2$, we arrive at
			\[
		1 -[\mu_1\|\widetilde u_n\|_2^2
		+\mu_2\|\widetilde v_n\|_2^2 ]=	\|\nabla\widetilde u_n\|_2^2
			+\|\nabla\widetilde v_n\|_2^2
			-\mu_1\|\widetilde u_n\|_2^2
			-\mu_2\|\widetilde v_n\|_2^2
			\longrightarrow0.
			\]
			Since
			$
			\widetilde u_n\to0,\
			\widetilde v_n\to0$ 
			in $L^2(\Omega),
			$
		it yields a contradiction. Therefore, $(u_n,v_n)$  is bounded in $\mathcal H.$
\end{proof}

In what follows, we present a version of the abstract linking theorem based on \cite[Theorem 6.3, Chapter II]{Struwe2008}, with $(PS)_c$ condition replaced by $(C)_c$ condition as in \cite{MS-ans}.

\begin{theorem}\label{thm:AL}
Let $V$ be a Banach space. Suppose $E \in \mathcal{C}^1(V)$ is even, that is $E(u) = E(-u)$, and satisfies $(C)_c$ condition for every $c\in \r^+$. Let $V^+, V^- \subset V$ be closed subspaces of $V$ with $\text{\rm codim } V^+ <\dim V^- < \infty$ and suppose $V = V^- + V^+$. Also suppose there holds 
	\[
	(1^\circ) \ E(0) = 0,
	\]
	\[
	(2^\circ) \ \exists \ \alpha > 0, \ \rho > 0 \ \forall \, u \in V^+ : \|u\| = \rho \implies E(u) \geq \alpha,
	\]
	\[
	(3^\circ) \ \exists R > 0 \ \forall \, u \in V^- : \|u\| \geq R \implies E(u) \leq 0.
	\]
	Then for each $j$, $1 \leq j \leq k = \dim V^- - \text{ \rm codim } V^+$, the numbers 
	\[
	c_j = \inf_{h \in \Gamma} \sup_{u \in V_j} E(h(u))
	\]
	are critical, where 
	\[
	\Gamma = \{h \in \mathcal{C}^0(V; V)\ : \ h \text{ is odd, and } h(u) = u \ \text{if } u \in V^- \text{ and } \|u\| \geq R\},
	\]
	and where $V_1 \subset V_2 \subset \dots \subset V_k = V^-$ are fixed subspaces of dimension 
	\[
	\dim V_j = \text{\rm codim }( V^+) + j.
	\]
	Moreover, 
	\[
	c_k \geq c_{k-1} \geq \dots \geq c_1 \geq \alpha.
	\]
\end{theorem}

Now, we are in the position to apply Theorem \ref{thm:AL} and prove Theorem \ref{t:existencebeta>1}.

\begin{proof}[Proof of Theorem \ref{t:existencebeta>1}]
We recall first that from Lemma \ref{eq:Pi}, any semi-trivial solution of problem \eqref{LS} has negative energy. Therefore, 
	any critical point to $J_\beta$, of  positive energy, is vectorial. In addition, since dim$(\cH^- \oplus \cH^0) =  \kappa_1 + \kappa_2$, then codim$ \cH^+ <  \kappa_1 + \kappa_2 +1 = $ dim$\cV_1$. Hence, from Lemmas \ref{lem:geo3} and \ref{lem:geo2} we obtain the linking geometry in Theorem \ref{thm:AL}, with 
	$E = J_\beta$, $V^+= \cH^+$,  $V^- = \cV_1$,  and  $V = V^- + V^+ =   \cV_1 + \cH^+  \supset \cV_0 + \cH^+= \cH$, in view of Lemma \ref{lem:geo0}.
	Furthermore, by Lemma \ref{lem:Cc}, we have that $J_\beta$ satisfies the $(C)_c$ condition for $c$ given in \eqref{c}. Since,  dim$\cV_1 - $ codim$ \cH^+ =1 $,  applying Theorem \ref{thm:AL} for $k=1$, we conclude that $c\geq\alpha >0$ is a critical level of $J_\beta$. Therefore, there exists at least one vectorial solution with positive energy. Finally, by Theorem \ref{t:nonexistencebeta>0}
	such a solution must have at least one component which changes sign.
\end{proof}

{\bf Acknowledgments:} The authors are grateful to the anonymous reviewer for the meticulous corrections and thoughtful suggestions, which have been invaluable in improving the clarity, rigor, and overall quality of the manuscript. The first author was supported  by FAPESP Proc 2022/15812-0, the second author was supported by CAPES, FAPDF and CNPq/Brazil, grant 300922/2025-6, and the third author was supported by CNPq/Brazil, grants 303.154/2025-0   and  152.786/2025-2.

\end{document}